\documentclass[10pt]{amsart}

\usepackage{amstext}
\usepackage{amssymb}
\usepackage{amsmath}
\usepackage{amsthm}
\usepackage{color}
\usepackage{multicol}
\usepackage{graphicx}
\DeclareGraphicsExtensions{.pdf,.png,.jpg}

\usepackage[english]{babel}

\DeclareRobustCommand{\stirling}{\genfrac\{\}{0pt}{}}

\theoremstyle{plain}
\newtheorem{theorem}{Theorem}[section]
\newtheorem{lemma}[theorem]{Lemma}
\newtheorem{prepos}[theorem]{Proposition}
\newtheorem{corol}[theorem]{Corollary}

\theoremstyle{definition}
\newtheorem{definition}[theorem]{Definition}

\newtheorem{remark}[theorem]{Remark}

\newtheorem{conjecture}{Conjecture}

\renewcommand{\Im}{\operatorname{Im}}
\renewcommand{\Re}{\operatorname{Re}}

\newcommand{\mesh}{\mathrm{mesh}}

\begin{document}

\author[O.~Katkova]{Olga Katkova}

\address{Department of Science and Mathematics, Wheelock College, USA}
\email{olga.m.katkova@gmail.com }

\author[M.~Tyaglov]{Mikhail Tyaglov}

\address{School of Mathematical Sciences, Shanghai Jiao Tong University}
\email{tyaglov@sjtu.edu.cn}

\author[A.~Vishnyakova]{Anna Vishnyakova}

\address{School of Mathematics and Computer Sciences, Kharkov National V.N.Karazin University}
\email{anna.m.vishnyakova@univer.kharkov.ua}




\title[Hermite-Poulain theorems]
{Hermite-Poulain theorems for linear finite difference operators}




\date{\small \today}



\begin{abstract}

We establish analogues of the Hermite-Poulain theorem for linear finite difference operators with constant coefficients defined on sets of polynomials with roots on a straight line, in a strip, or in a half-plane. We also consider the central finite difference operator of the form
$$
\Delta_{\theta, h}(f)(z)=e^{i\theta}f(z+ih)-e^{-i\theta}f(z-ih), \quad\theta\in[0,\pi),\ \ h\in\mathbb{C}\setminus\{0\},
$$
where $f$ is a polynomial or an entire function of a certain kind, and prove that the roots of $\Delta_{\theta, h}(f)$ are simple under some conditions. Moreover, we prove that the operator $\Delta_{\theta, h}$
does not decrease the mesh on the set of polynomials with roots on a line and find the minimal mesh. The asymptotics of the roots of $\Delta_{\theta, h}(p)$ as $|h|\to\infty$ is found for any complex polynomial $p$.   Some other interesting roots preserving properties of the operator $\Delta_{\theta, h}$ are also studied, and a few examples are presented.
\end{abstract}

\maketitle


\vspace{2cm}

\setcounter{equation}{0}
\section{Introduction}

One of the most important problems in the zero distribution theory of polynomials and transcendental entire functions 
is to describe linear transformations that map polynomials with all roots in a given area into the set of polynomial with all roots in another given area.
One of the first researchers who started to study such type of problems systematically were Hermite and, later, Laguerre who considered 
linear operator preserving the set $\mathcal{HP}$ of hyperbolic polynomials, that is, real polynomials with only real roots.

In 1914~\cite{polsch}, P\'olya and Schur completely described the operators acting diagonally on the standard monomial
basis $1$, $x$, $x^2$, \dots of $\mathbb{R}[x]$ and preserving $\mathcal{HP}$.

Later the study of linear
transformations sending real-rooted polynomials to real-rooted polynomials was continued by many authors including
N.\,Obreschkov, S.\,Karlin, B.\,Levin, G.\,Csordas, T.\,Craven, K.\,de Boor, R.\,Varga, A.\,Iserles, S.\,N{\o}rsett, E.\,Saff etc.
Among recent authors it is especially worth to  mention P. Br\"{a}nd\'{e}n and J. Borcea~\cite{BrandenBorcea} (see also~\cite{BrandenBorcea2,BrandenBorcea3}), who completely
characterized all linear operators preserving real-rootedness of real polynomials (and some other root location preservers).

The present work was inspired by the paper~\cite{BKS}, where the authors made an attempt to transfer
the existing theory of real-rootedness preservers to the basis of Pochhammer symbols and to develop a finite difference analogue of the P\'olya-Schur theory. Some results on this topic were also obtained in~\cite{KSV}. But both works deal with finite  difference operators with real step defined on a subspace of $\mathcal{HP}$.

A natural extension of
polynomials with real roots is the so-called Laguerre-P\'olya class.

\begin{definition}\label{th:d2} 
	A real entire function $f$ is said to be in the {\it
Laguerre-P\'olya class}, denoted as $\mathcal{LP}$, if
\begin{equation}\label{e2}
 f(z) = c z^n e^{- a z^2+b z}\prod_{k=1}^\infty
\left(1-\frac {z}{x_k} \right)e^{\tfrac {z}{x_k}},
\end{equation}
where $c, b, x_k \in  \mathbb{R}$,  $x_k\ne 0$,  $a \geqslant 0$,
$n$ is a non-negative integer and $\sum\limits_{k=1}^\infty x_k^{-2} <
\infty$. The product in the right-hand side of~\eqref{e2} can be
finite or empty (in the latter case the product equals 1).
\end{definition}

This class is essential in the theory of entire functions due to the fact that these
and only these functions  can be uniformly approximated on compact subsets of 
$\mathbb{C}$ by a sequence of real polynomials with only real zeros.
For various properties and
characterizations of the
Laguerre-P\'olya class see, e.g., \cite[p. 100]{pol}, \cite{polsch}, \cite[Chapter VII]{lev},
\cite[pp. 42--47]{HW},  \cite[Kapitel II]{O} or  \cite{CV}.

One of the first results on  $\mathcal{LP}$-preservation properties of
linear finite difference operators was obtained by G.\,P\'olya.
In~\cite{pol1}, he established that if
$f \in \mathcal{LP},$ then  $f(x+ih) + f(x-ih)\in \mathcal{LP}$
for every $h \in \mathbb{R}$. N.G.\,de Bruijn observed that this fact can be refined as follows, cf.,~\cite[Theorem~3]{Br}.
\begin{theorem}[de Bruijn]\label{Theorem.de_Bruijn}
For any $h\in \mathbb{R}$ and $\xi\in\mathbb{C}\setminus\{0\}$, the
function
\begin{equation*}\label{e33}
\xi f(x+ih) + \overline{\xi} f(x-ih)
\end{equation*}
belongs to the class $\mathcal{LP}$ whenever $f\in\mathcal{LP}$. Here $\overline{\xi}$ means the complex conjugate to $\xi$.
\end{theorem}

One of the main goals of the present work is to extend this result of N.G.\,de Bruijn and to establish an analogue of
the so-called Hermite-Poulain theorem, see, e.g., \cite[p.~4]{O} or~\cite[Part~3, Ch.~1, no.~35]{PS1}, claiming that a finite order linear differential operator $T = a_0+a_1d/dx+\cdots+a_kd^k/dx_k$ with
constant coefficients is hyperbolicity preserver if and only if its symbol polynomial $Q_T(t)=a_0+a_1t+\cdots+a_kt^k$ is hyperbolic.

In this work, we do not restrict ourselves by real or pure imaginary steps of finite differences, and consider
somewhat general central finite difference operators with arbitrary nonzero complex steps:
\begin{equation}\label{Delta.general}
\Delta_{\alpha,h}f(x)=\alpha f(x+h)-\alpha^{-1} f(x-h),
\end{equation}
where $\alpha,h\in\mathbb{C}\setminus\{0\}$. We are interested in root location of the function
$\Delta_{\alpha,h}f(x)$ with respect to the root location of the polynomial (or entire
function) $f$. 
We also consider  compositions of operators of the form~\eqref{Delta.general}




If all roots of a polynomial $p$ lie on the line
\begin{equation*}\label{line}
L_{\varphi,c}=\{ae^{i\varphi}+c,\ a\in\mathbb{R}\},
\end{equation*}
where $\varphi\in[0,\pi)$ and $c\in\mathbb{C}$ are fixed numbers, then the roots of the polynomial $\Delta_{\alpha,h}(p)$ lie on the same line if and only if
$|\alpha|=1$, and $\arg h=\varphi\pm\dfrac{\pi}2$. This fact allows us to establish a finite difference analogue of the Hermite-Poulain theorem mentioned above and even a more general fact.

\begin{theorem}\label{th:mth1}
Let $L_{\varphi_1,c_1}=\{ae^{i\varphi_1}+c_1,\ a\in\mathbb{R}\}$ and $L_{\varphi_2,c_2}=\{ae^{i\varphi_2}+c_2,\ a\in\mathbb{R}\}$,
$c_1,c_2\in\mathbb{C}$, $\varphi_1,\varphi_2\in[0,\pi)$, be two lines on the complex plane.
The operator
\begin{equation}\label{e3}
T(p) (x) = \sum_{k=l}^m a_k p(x-kh),\qquad a_ma_l\neq0,
\end{equation}
sends any polynomial with zeros on the line $L_{\varphi_1,c_1}$ to a polynomial with zeros on the line $L_{\varphi_2,c_2}$ if and only if the following conditions hold:
\begin{itemize}
\item[1)] $\varphi_1=\varphi_2=\varphi$ for some $\varphi\in[0,\pi)$;
\item[2)] $(l+m)\,h=2\Im\left(e^{-i\varphi}(c_2-c_1)\right)\cdot e^{i\left(\tfrac{\pi}2+\varphi\right)}$;
\item[3)] All the zeros of the generating function
\begin{equation}\label{g1} 
Q (t) = \sum_{k=l}^m a_k t^k.
\end{equation}
lie on the unit circle $\mathbb{T}:=\{ z\in\mathbb{C} : |z| = 1 \}$;
\end{itemize}
\end{theorem}

Thus, the operator $T$ defined in~\eqref{e3} can preserve a class of polynomials with zeros on a line if and only if $l=-m$ and $h=\alpha e^{i\left(\varphi+\tfrac{\pi}2\right)}$, $\alpha\in\mathbb{R}\setminus\{0\}$. In particular, the following fact is true.

\begin{corol}\label{Corol.Hermite.Poulain}
The linear operator $T$ of the form \eqref{e3} preserves the set of 
hyperbolic polynomials if and only if the following 
conditions hold
\begin{itemize}
	\item [1)] $ \Re h =0$ ;

     \item[2)] $l = - m$;

	\item [3)] All roots of the generating function (\ref{g1}) belong to the unit circle 
$\mathbb{D}$;
\item[4)] $a_{m}\cdot a_{-m}>0$.
\end{itemize}
\end{corol}

\begin{remark}
The conditions 2) and 3) of Corollary~\ref{Corol.Hermite.Poulain} mean that the generating function \eqref{g1} has the form 
$$ Q(t)= C\prod_{k=1}^{2m}\left(e^{-i\theta_k/2}\sqrt{t}+e^{i\theta_k/2}\frac{1}{\sqrt{t}} \right),$$
where the numbers $C$ and $ \theta_k \   ( k=1, 2, \ldots,  2m)$  are real. Thus, Corollary~\ref{Corol.Hermite.Poulain} states 
that every linear operator of the form \eqref{e3} that preserves the set of  hyperbolic polynomials is a 
composition of  linear operators of the form \eqref{e33} with $|\alpha|=1$.

The condition $4)$ in Corollary~\ref{Corol.Hermite.Poulain} provides the reality of the coefficients of $T(p)$.  
\end{remark}

Note also that in Theorem~\ref{th:mth1} as well as in Corollary~\ref{Corol.Hermite.Poulain}, the step of the finite difference operator $T$ must be orthogonal to the line where the roots of the given polynomial lie. This
condition differs from the conditions of an analogue of the Hermite-Poulain theorem established in~\cite{BKS}, where the authors considered the operator~\eqref{e3} with a real step undertaking an attempt to transfer the existing
theory of real-rootedness preservers to the basis of Pochhammer symbols and to develop a 
finite difference analogue of the P\'olya-Schur theory. They proved that the operator $T$ preserves a specific subset of the set of hyperbolic polynomials if and only if its generating function~\eqref{g1} has only nonnegative  roots. They also found out that
a linear operator of the form
$
T(p) (x) = \sum\limits_{k=l}^m c_k(x) p(x-kh),
$
with polynomials coefficients~$c_k(x)$, $k=l,\ldots,m$, and with a real step $h$ preserves the set of hyperbolic polynomials 
if and only if at most one of coefficients $c_k(x)$ is not identically zero, and $c_k(x)$ is hyperbolic for such a $k$, see~\cite{BKS} for more details. 

\vspace{2mm}

The fact that every linear operator of the form~\eqref{e3} that preserves the set of  hyperbolic polynomials is a 
composition of  linear operators of the form~\eqref{Delta.general} motivates us to study such kind of operators  in more detail.  Let us put $|\alpha|=1$ in~\eqref{Delta.general}. Then without loss of generality one can consider the operator
\begin{equation}\label{Delta.theta}
\Delta_{\theta,h}f(x)=\dfrac{e^{i\theta} f(x+ih)-e^{-i\theta} f(x-ih)}{2i},
\end{equation}
where $\theta\in\left[0,\pi\right)$ and $\arg h=\varphi$ for some $\varphi\in[0,\pi)$. If the roots of the polynomial $f$
lie on the line $L_{\varphi,c}$, then the roots of $\Delta_{\theta,h}(f)$ not only lie on the same line $L_{\varphi,c}$ but also are of \textit{multiplicity one} (in fact, we prove more, see~Theorem~\ref{th:simplicity.poly}).
And this property holds in the case when $f$ is an entire function from the closure of the set of polynomials with roots on a line $L_{\varphi,c}$ (Theorem~\ref{th:simplicity.ent.func}).

Note that the fact that the operator $\Delta_{0,h}$ with $h\in\mathbb{R}\setminus\{0\}$ preserves the real-rootedness of polynomials and entire functions
was established by P\'olya in~\cite[Hilfssatz II]{pol1} (see also Theorem~\ref{Theorem.de_Bruijn}). However, the simplicity of the roots of $\Delta_{\theta,h}(f)$ seems to be a new property of finite differences.
In particular, we have the following extenstion of the P\'olya-de Bruijn theorem.

\begin{theorem}\label{th:mth3}
Let $h\in\mathbb{R}\setminus\{0\}$, $\theta \in[0,2\pi) $. For any $f \in \mathcal{LP}$ , all the zeros of $\Delta_{\theta,h}(f)$ are real and simple.
\end{theorem}

For example, if we consider the polynomial $p(z)=z^n$ with only one (multiple) root, then the roots $\lambda_k(\theta)$ of the polynomial $\Delta_{\theta,1}(p)$ are simple and lie
on the real line $\mathbb{R}$. In fact, the roots have the form (see Lemma~\ref{Lemma.roots.Delta.x^n}):

\vspace{2mm}

\noindent for $\theta\in\left(0,\pi\right)$, 
\begin{equation}\label{cot.root.general}
\lambda_{k,n}(\theta)=\cot\dfrac{\pi k-\theta}{n},\qquad k=1,\ldots,n,
\end{equation}

\noindent and for $\theta=0$,
\begin{equation}\label{cot.root.0}
\lambda_{k,n}(0)=\cot\dfrac{\pi k}{n},\qquad k=1,\ldots,n-1.
\end{equation}
The polynomials 
\begin{equation}\label{poly.Qn}
Q_n(x):=\Delta_{\theta,1}(x^n),\qquad n\in\mathbb{N},
\end{equation}
and their roots $\lambda_{k,n}(\theta)$ play a very important role in our study.
For instance, they appear in the estimates of the minimal and maximal roots of $\Delta_{\theta,h}(p)$ for $p\in\mathcal{HP}$.

\vspace{2mm}

Let us denote by $\mu_{max}(p)$ and $\mu_{min}(p)$ the {\it maximal} and the {\it minimal} roots of a given hyperbolic polynomial $p$, respectively. Then the following theorem holds.
\begin{theorem}\label{th:mth5}
For any  $p \in \mathcal{HP}$ , $\deg p =n \geqslant 1$,  $\theta\in[0,\pi),  $ 
and for any $h>0$, the following inequalities hold 
$$
\mu_{max}(\Delta_{\theta, h}(p)) \leqslant \mu_{max}(p)+h\cdot \mu_{max}(Q_n) \quad\text{and}\quad 
\mu_{min}(\Delta_{\theta, h}(p)) \geqslant \mu_{min}(p)+h\cdot \mu_{min}(Q_n).
$$
\end{theorem}

\noindent From~\eqref{cot.root.general}--\eqref{cot.root.0} it follows that

for $\theta\in(0,\pi)$,
\begin{equation}\label{min.max.Qn.zeros}
\mu_{min}(Q_n)=\lambda_{n,n}(\theta)=-\cot\dfrac{\theta}{n}\quad
\text{and}\quad
\mu_{max}(Q_n)=\lambda_{1,n}(\theta)=\cot\dfrac{\pi-\theta}{n},
\end{equation}

and for $\theta=0$,
\begin{equation}\label{min.max.Qn.zeros.0}
\mu_{min}(Q_n)=
\lambda_{n-1,n}(0)=-\cot\dfrac{\pi}{n}\quad
\text{and}\quad
\mu_{max}(Q_n)=\lambda_{1,n}(0)=\cot\dfrac{\pi}{n}.
\end{equation}

\begin{definition}\label{th:d4}
Given a hyperbolic polynomial $p$ denote by $\mesh(p)$ its mesh, i.e. the minimal distance between its roots
$$
\mathrm{mesh} (p) := \min\limits_{1\leq j\leq n-1} (x_{j+1}-x_j),
$$
where $p=C(x-x_1)(x- x_2) \cdots(x- x_n)$, $C\neq0$, and  $x_1 \leq x_2 \leq \ldots \leq x_n$. 

If $p$ has a multiple root, then $\mesh(p)=0$ by definition. Polynomials of degree at most $1$ are defined to have mesh equal to $+\infty$.
  
\end{definition}

Since by Theorem~\ref{th:simplicity.poly} the roots of the polynomial $\Delta_{\theta,h}(p)$ are simple for any polynomial $p\in\mathcal{HP}$, its mesh is positive
$\mesh\,\Delta_{\theta,h}(p)>0$, and one can pose a question of the minimal value of $\mesh\,\Delta_{\theta,h}(p)$. The answer to this question
is provided by the following theorem.
\begin{theorem}\label{th:mth4}
Let  a polynomial $p$ of degree $n$ is hyperbolic, and let $h>0$. Then the inequality 
\begin{equation}\label{mesh.main.ineq}
\mathrm{mesh}\,\Delta_{\theta, h}(p)\geqslant\max\{\mathrm{mesh}\,p,\mathrm{mesh}\,\Delta_{\theta, h}(x^n)\}
\geqslant\mathrm{mesh}\,\Delta_{\theta, h}(x^n)
\end{equation}
holds for any $\theta\in(0,\pi)$ whenever $n\geqslant2$.

If $ \theta = 0$, the inequality~\eqref{mesh.main.ineq} is true whenever
 $n\geqslant3$.
\end{theorem}
In other words, the operator $\Delta_{\theta, h}$ does not decrease the mesh, and the minimal mesh in the image 
of this operator on the set $\mathbb{R}^n[x]$ is $\mesh\,\Delta_{\theta, h}(x^n)$. That is, 
 if for a hyperbolic polynomial $q$ of degree $n$, its mesh is less than the mesh of the
polynomial $\Delta_{\theta, h}(x^n)$, then $q$ cannot be represented as $\Delta_{\theta,h}(p)$ for a polynomial~$p\in\mathcal{HP}$. Thus, the minimal mesh for the image of the operator $\Delta_{\theta,h}$ with 
$$
\theta=\dfrac{\pi}2\pm\dfrac{\psi}{2},\qquad\psi\in[0,\pi],
$$ 
can be easily calculated:
\begin{equation}\label{min.mesh}
\mesh\,\Delta_{\theta,h}(x^n)=h\cdot
\begin{cases}
&\ \dfrac{\sin\tfrac{\pi}{n}}{\cos\tfrac{\pi-\psi}{2n}\cdot\cos\tfrac{\pi+\psi}{n}}\ \ \,\qquad\qquad\text{if}\quad n\ \ \text{is even},\\
&\\
&\dfrac{\sin\tfrac{\pi}{n}}{\cos\tfrac{\psi}{2n}\cdot\cos\tfrac{2\pi-\psi}{2n}}
\ \ \,\quad\qquad\text{if}\quad n\ \ \text{is odd}.
\end{cases}
\end{equation}
Thus, for example, if the mesh of a hyperbolic polynomial $p$ of degree $4$ is less than $1$, then it cannot be represented as $q(z)=p(z+i)-p(z-i)$ for some $q\in\mathcal{HP}$. It would be interesting to investigate the image of the operator $\Delta_{\theta, h}$ defined on $\mathcal{HP}$ in detail. This topic is discussed in  Section~\ref{section:open.problems}.

We remind that one of the first results about the mesh of polynomials was established by M.Riesz in 1925 whose elementary proof was given by A.\,Stoyanoff in~\cite{Sto}. M. Riesz proved that the operator of differentiation on $\mathcal{HP}$ does not decrease the mesh (see Theorem~\ref{Theorem.Riesz} of the present work). Moreover, in \cite[p. 226, Lemma 8.25]{Fisk} S. Fisk established that any linear operator defined on $\mathcal{HP}$ does not decrease the mesh if it commutes with any shift operator with real step (see Theorem~\ref{Theorem.Fisk} for more details). This fact implies Theorem~\ref{th:mth4} partially, since the operator~\eqref{Delta.theta} obviously commutes with any shift operator, see Section~\ref{section:minimal.mesh}. At the same time, Theorem~\ref{th:mth4} not only shows that 
$\Delta_{\theta, h}$ does not decrease the mesh but also provides a sharp lower bound of the mesh of the image of the operator $\Delta_{\theta, h}$.

As a by-product of Theorem~\ref{th:mth1}, we establish analogues of the Hermite-Poulain theorem for polynomials with roots in strips and half-planes. 
\begin{theorem}\label{th:mth2}
	Let $\mathcal{S}$ be a closed strip in the complex plane $\mathbb{C}$ bounded by lines $L_{\varphi,c_1}$ and $L_{\varphi,c_2}$, $c_1\neq c_2$. The operator~\eqref{e3} preserves the set of polynomials with roots in $\mathcal{S}$ if and only if the following conditions hold
	\begin{itemize}
		\item[1)] $\arg h=\dfrac{\pi}{2}+\varphi$;
		\item[2)] $l = - m$; 
		\item[3)] All the zeros of the generating function~\eqref{g1}
		        lie on the unit circle $\mathbb{D}$.
	\end{itemize} 
    If additionally the width of the strip does not exceed $\dfrac{|h|}2$, the roots 
    of the polynomial  $T(p)$ are simple for any polynomial $p$ with roots in $\mathcal{S}$.
\end{theorem}

Note that the sufficiency of conditions $1)$--$3)$ in Theorem~\ref{th:mth2} for \textit{real} polynomials with roots in a strip $\mathcal{S}_r:=\{z\in\mathbb{C}\ :\ |\Im z|\leqslant r\}$  was proved by N.G.\,de Bruijn in \cite{Br}. In fact, he proved (not only for polynomials but for entire functions as well) that 
given a \textit{real} polynomial (or certain \textit{real} entire function) with roots in $\mathcal{S}_r$ , the roots of $T(p)$ lie in a more narrow strip
$\{z\in\mathbb{C}\ :\ |\Im z|\leqslant \sqrt{r^2-mh^2}\}$ if $r>h\sqrt{m}$, or are real if $r\leqslant h\sqrt{m}$ provided the conditions $1)$--$3)$ of Theorem~\ref{th:mth2} hold, see~\cite[Theorems~4~and~8]{Br}. A~similar result for the operator $\Delta_{\theta, h}$ defined in~\eqref{Delta.theta} was obtained in~\cite{Cardon} for entire functions by a technique different from the one used in~\cite{Br}. Moreover, in~\cite{BrandenChasse} there was a complete characterization of 
all linear operators which preserve certain spaces of entire functions whose zeros lie in a closed strip. However, in the present work, instead of applying the result of~\cite{BrandenChasse} we use another technique which is more close to the one used in the classical papers~\cite{Br,pol1}, and, from our point view, is more convenient to finite difference operators we consider. In addition, by our technique we are able to study the multiplicities of the roots of the functions~$T(p)$ and $\Delta_{\theta, h}(p)$.  We also pay attention of the reader to the work~\cite{CsorSmith},
where the authors considered a subclass of real entire functions~$f$ with zeros in a strip such that the function $\Delta_{\theta,ih}(f)$, $h<0$, has infinitely many pure imaginary zeros. 

It is clear that zero strip preservers are also zero half-plane preservers.
\begin{corol}\label{Corol:mth2.}
		Let $\mathcal{H}$ be a closed half-plane in the complex plane $\mathbb{C}$ bounded by a line $L_{\varphi,c}$. The operator~\eqref{e3} preserves the set of polynomials with roots in $\mathcal{H}$ if and only if the conditions $1)$--$3)$ of Theorem~\ref{th:mth2} hold.
\end{corol}
For the case when $T=\Delta_{\theta,h}$, the sufficiency statement of Corollary~\ref{Corol:mth2.} is a particular case of a result by L.\,Kuipers~\cite{Kuipers} (see also~\cite[Theorem~2.2.1]{Rahman_Scm}).

Finally, we notice once again that from the  aforementioned results by N.G.\,de Bruijn it follows that for any \textit{real} polynomial with roots in a strip $\mathcal{S}_b$, the polynomial $\Delta_{\theta, h}(p)$ has only real roots if the step $h$ of the finite difference operator is sufficiently large. However, as is easy to check, this is not true for complex polynomials with non-real roots non-symmetric w.r.t. the real line. In this case, the polynomial $\Delta_{\theta, h}(p)$ may have non-real roots for any finite $h>0$. This suggested us to find the asymptotics of the roots of $\Delta_{\theta, h}(p)$ as $h\to+\infty$ for arbitrary complex polynomial~$p$.
\begin{theorem}\label{th:mth6} 
Let
$$
p(z)=a_0z^n+a_1z^{n-1}+\cdots+a_{n-1}z+a_n,\qquad a_k\in\mathbb{C},\  a_0\neq0,
$$
be an arbitrary complex polynomial. The $k$-th root of the polynomial
$\Delta_{\theta, h}(p)$ satisfies the following asymptotic formula
\begin{equation*}\label{asympt.formula}
\mu_k(\theta,h)=h\cdot\lambda_{k,n}(\theta)-\dfrac{a_1}{n\,a_0}-\dfrac{Q_n''\left(\lambda_{k,n}(\theta)\right)}{n!\,Q_n'\left(\lambda_{k,n}(\theta)\right)}\cdot
\dfrac{p^{(n-2)}\left(-\tfrac{a_1}{n\,a_0}\right)}{a_0h}-\dfrac{Q_n'''\left(\lambda_{k,n}(\theta)\right)}{n!\,Q_n'\left(\lambda_{k,n}(\theta)\right)}\cdot
\dfrac{p^{(n-3)}\left(-\tfrac{a_1}{n\,a_0}\right)}{a_0h^2}\,+\,
O\left(\frac{1}{h^3}\right)
\end{equation*}
as $|h|\to\infty$, where the polynomial $Q_n(z)=\Delta_{\theta,1}(z^n)$, and $\lambda_{k,n}(\theta)$ are its roots defined in~\eqref{cot.root.general} and~\eqref{cot.root.0}.
\end{theorem}

For example, for the polynomial $p(x)=\displaystyle\prod_{k=1}^{12}(x-k)\prod_{k=1}^{32}(x+ki)$ of degree $44$, the roots of the polynomial $\Delta_{0,h}p$ (classical central finite difference) with
$h=10\,e^{\tfrac{\pi}3i}$ and $h=50\,e^{\tfrac{\pi}3i}$ are depicted
on Figures~\ref{fig:pic_1} and~\ref{fig:pic_2}, respectively.

\begin{figure}[h]
\begin{multicols}{2}
\hfill
\includegraphics[scale=0.4]{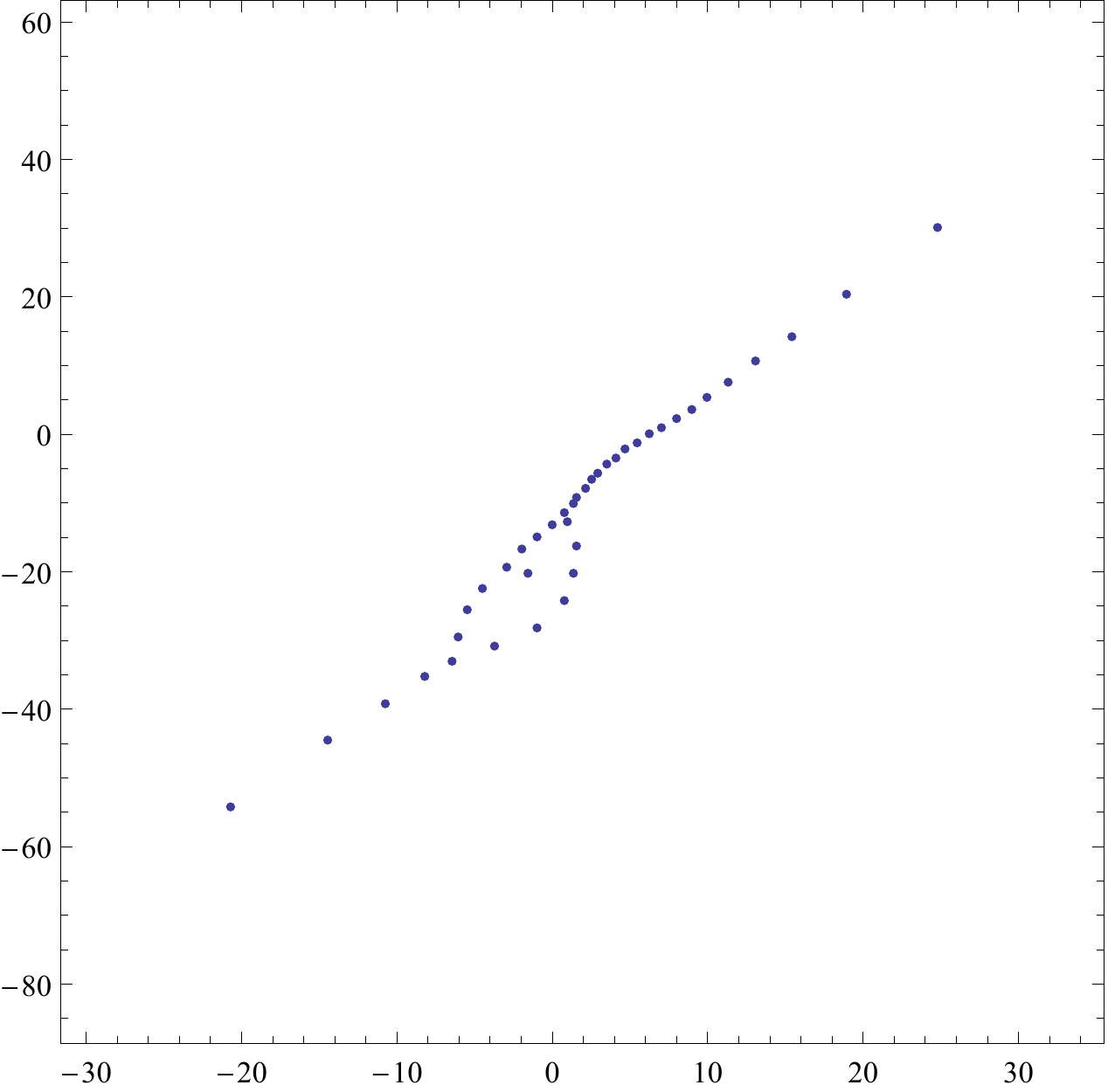}
\hfill\hfill
\caption{$h=10\,e^{\tfrac{\pi}3i}$}
\label{fig:pic_1}
\hfill
\includegraphics[scale=0.4]{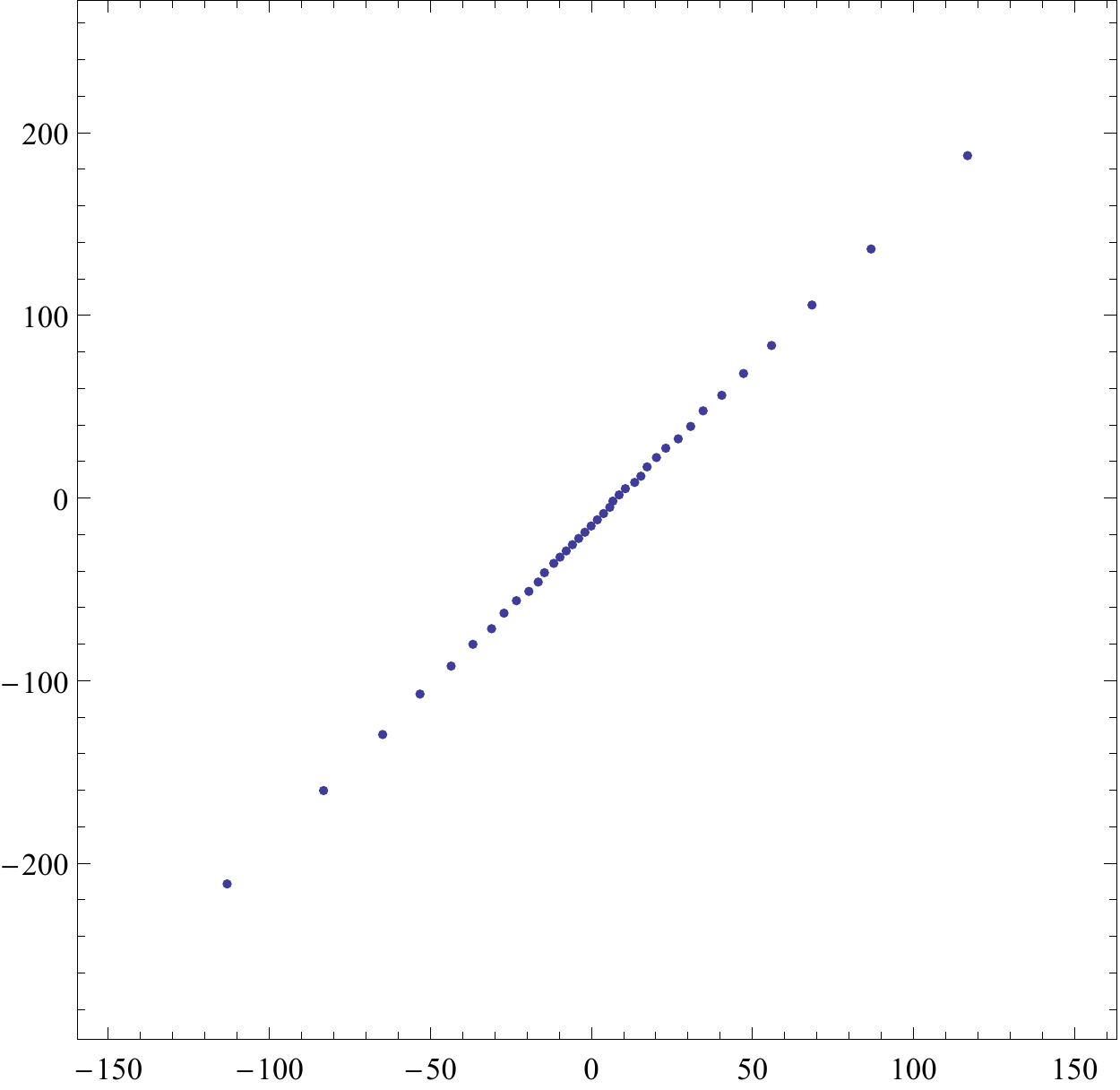}
\hfill\hfill
\caption{$h=50\,e^{\tfrac{\pi}3i}$}
\label{fig:pic_2}
\end{multicols}
\end{figure}

From the formula~\eqref{asympt.formula}, it follows that for sufficiently large step $h$ the roots of $\Delta_{\theta, h}(p)$ tend to a line parallel to $h$. This means that for large $h$, $\Delta_{\theta, h}$ is a zero strip decreasing operator. Calculations show that this operator decreases the strip of roots location for any $h$. However, we do not have a strict proof of this fact, see Section~\ref{section:open.problems} for details.

What is known in this sense is that for pure imaginary step ($h\in\mathbb{R}$), the operator $\Delta_{\theta,h}$ is a CZD (complex zero decreasing) operator (see~\cite{CsorCrav}), that is, the number
of non-real roots of $\Delta_{\theta,h}(p)$ does not exceed the number of non-real roots of $p$. As was noted in~\cite[p.~200]{Br}, this simple fact is a special case of the Hermite-Poulain theorem mentioned above.

The operator $\Delta_{\theta,h}$  also preserves the class of the so-called self-interlacing polynomials (see Theorem~\ref{Theorem.SI}). We remind that a polynomial $p(z)$ is called self-interlacing if all its roots are real, simple, and interlacing with the roots of $p(-z)$.

\vspace{2mm}

The rest of the paper is organized as follows. In Section~\ref{section:polynomials.line} we study properties of linear finite difference operators defined on the set of polynomials with roots on a straight line and prove Theorem~\ref{th:mth1}. In Section~\ref{section:entire_functions} we consider a specific set of entire functions with roots on a line and establish Theorem~\ref{th:mth3} and its generalization Theorem~\ref{th:simplicity.ent.func}. 
Sections~\ref{section:polynomials.line} and~\ref{section:entire_functions} are illustrated by some examples from combinatorics and the functions theory, see Proposition~\ref{Corollary.zeroes. of.finite.differences} and Theorem~\ref{Theorem.reciprocal.gamma}. Section~\ref{section:minimal.mesh} is devoted to mesh of polynomials with roots on a line. Here we prove formul\ae~\eqref{cot.root.general}--\eqref{cot.root.0} and~\eqref{min.mesh} as well as Theorems \ref{th:mth5} and~\ref{th:mth4}. In Section~\ref{section:polynomials.strip} we deal with polynomials with roots in a strip or in a half-plain and prove Theorem~\ref{th:mth2}. The asymptotics of the roots of $\Delta_{\theta, h}(p)$ as $|h|\to\infty$ is studied in Section~\ref{section:asymptotics} where we prove Theorem~\ref{th:mth6}.  Finally, in Section~\ref{section:open.problems} we discuss open question related to the considered theory of linear finite differences operators.

\setcounter{equation}{0}
\section{Hermit-Poulain theorem for polynomials with roots on a straight line}\label{section:polynomials.line}

In this section, we prove Theorem~\ref{th:mth1} and Corollary~\ref{Corol.Hermite.Poulain}. We also prove that
if the operator~\eqref{e3} satisfies the conditions of Theorem~\ref{th:mth1}, then all polynomials in its image have only simple roots.

Consider two straight lines in the complex plane
\begin{equation*}
L_{\varphi_j,c_j}=\{ae^{i\varphi_j}+c,\ a\in\mathbb{R}\},\quad j=1,2,
\end{equation*} 
$c_j\in\mathbb{C}$, $\varphi_j\in(0,\pi)$, and consider the operator~\eqref{e3} having the form
\begin{equation}\label{e3.double}
T(p) (x) = \sum_{k=l}^m a_k p(x-kh),\qquad a_ma_l\neq0.
\end{equation}
First we prove the following fact.
\begin{prepos}\label{Propos.parallel}
	Suppose that for every polynomial $p$ with roots on the line $L_{\varphi_1,c_1}$, all the roots of $T(p)$ lie on the line $L_{\varphi_2,c_2}$, then $\varphi_2=\varphi_1$.
\end{prepos}
\begin{proof}
Without loss of generality one can assume that $L_{\varphi_1,c_1}$ is the real line, $\varphi_1=0$ and $c_1=0$, since the general case can be obtain from the case $L_{\varphi_1,c_1}=\mathbb{R}$ by changing of variables. Obviously, the operator $T$ does not map all hyperbolic polynomials into polynomials of degree 0.

Suppose that $p$ is a hyperbolic polynomial such that $T(p)(z_0) = 0$, so
$z_0\in  L_{\varphi_2,c_2}$ by assumption of the proposition. For $a \in\mathbb{R}$, define the polynomial 
$$
p_a(z):=p(z-a),
$$ 
which is hyperbolic as well, so we have $T(p_a)(z) = T(p)(z -a)$, and $T(p_a)(z_0+a)=0$ by assumption. Consequently, for every $a\in\mathbb{R}$, we obtain $z_0 +a\in L_{\varphi_2,c_2}$,  that is possible only if the line $L_{\varphi_2,c_2}$ is parallel to the real line, that is, if $\varphi_2=0$, as required.
\end{proof}

Now we are in a position to describe the root location of the generating function $Q$ defined in~\eqref{g1}.

\begin{prepos}\label{Propos.unit.circle}
	Let the operator $T$ defined in~\eqref{e3.double} sends every polynomial with roots on the line $L_{\varphi,c_1}$ to a polynomial with roots on the line $L_{\varphi,c_2}$. Then all zeros of its generating function $Q$ lie on the unit circle, and the step $h$ is orthogonal to the lines $L_{\varphi,c_1}$ and $L_{\varphi,c_2}$, that is, $\arg h=\varphi\pm\dfrac{\pi}{2}$.
\end{prepos}
\begin{proof}
First, suppose that	$L_{\varphi,c_1}=\mathbb{R}$, that is, $\varphi=0$ and $c_1=0$.

Thus, if $p$ is  a hyperbolic polynomial, then by assumption the roots of $T(p)$ lie on the line 
\begin{equation}\label{line.b}
L_{0,c_2}=\{z\in\mathbb{C}\ :\ \Im z=\Im c_2\},
\end{equation}
parallel to the real axis.

Consider the polynomial $P_n (x) =x^n$,   $n\in \mathbb{N}$.  For $P_n(x)$ we have
\begin{equation}\label{f2}
T(P_n) (x) = \sum\limits_{k=l}^m a_k (x-k h)^n = x^n 
\sum_{k=l}^m a_k \left(1-\dfrac{k h}{x}\right)^n =:
x^n R_n(x).
\end{equation} 
By assumption, all the zeros of the rational function $R_n$ lie on the line~\eqref{line.b}. 
Then the zeros of the polynomial 
$$
F_n(y):=R_n \left(\dfrac{n}{y}\right)
$$
lie on the circle
\begin{equation*}
C_n=\left\{z\in\mathbb{C}\ :\ \left|z+i\,\dfrac{n}{2\Im c_2}\right|=\left|\dfrac{n}{2\Im c_2}\right|\right\}.
\end{equation*}

As $n\to \infty$, the sequence of polynomials $\{F_n(y)\}_{n=1}^{\infty}$ converges 
uniformly on compact sets to the following entire function 
$$
f(y) := \sum\limits_{k=l}^m 
	a_k e^{-j h y} = Q (e^{- h y}).
$$
Each zero of the function $f$ is the accumulation point of a sequence of zeros of $F_n$. It is clear that  if a sequence $\{ z_k\}_{k\in \mathbb{N}}$ has a limit $y_0$ and  
if	$z_k \in C_k $, $\forall k\in \mathbb{N}$,  then $y_0$  is real. 
Consequently, all the zeros of the function $f$ are real. 

Suppose now that $z_0 \in \mathbb{C}\setminus \{0\}$ is a zero of the generating function $Q$ defined in~\eqref{g1}, and let 
$$
h=\alpha+i \beta,\quad \alpha, \beta \in \mathbb{R}.
$$
Solving the equation
$$
e^{-h y} = z_0
$$
for real $y$, we obtain
$$
y_k = - \frac{\log |z_0| +
	i\arg z_0 + 2\pi k i }{\alpha + i \beta}, \qquad   k\in \mathbb{Z}
$$
that implies
$$
\Im y_k =\dfrac{\beta \log |z_0| - \alpha \arg
	z_0 - 2\pi k \alpha }{\alpha^2 +  \beta^2}=0,
$$
for any $k\in \mathbb{Z}$, that is possible only if $\alpha=\Re h=0$ and
$|z_0|=1$.

Thus, we proved that if the operator $T$ sends a hyperbolic polynomial to a polynomial with roots on a line $L_{0,c_2}$ parallel to the real axis, then the roots of its generating function $Q$ lie on the unit circle, and the step $h$ is pure imaginary, that is, orthogonal to the line $L_{0,c_2}$ and to the real axis.

For the general case, when $\varphi\in[0,\pi)$ and $c_1$ is an arbitrary complex number, the assertion of the proposition follows from the previous result by changing variables.
\end{proof}

To prove the formula for the step in Condition $2)$ of Theorem~\ref{th:mth1}, let us consider the shift operator
\begin{equation}\label{shift.operator}
S_\lambda(p)(x) :=p(x-\lambda),\qquad\lambda\in\mathbb{C}.
\end{equation}
acting on $\mathbb{C}[z]$.

If the operator $T$ defined in~\eqref{e3.double} sends every polynomial with roots on a line $L_{\varphi_1,c_1}$ to a polynomial with roots on a line $L_{\varphi_2,c_2}$, then by Propositions~\ref{Propos.parallel}--\ref{Propos.unit.circle}, $\varphi_1=\varphi_2=:\varphi$, $\arg h=\varphi\pm\dfrac{\pi}2$, and all the roots of the generating function $Q$ lie on the unit circle, that is,
\begin{equation*}
Q(t) = \sum\limits_{k=l}^m a_k t^{k} = t^l \prod\limits_{j=1}^{m-l}(t- e^{i\theta_j}),\qquad\theta_j\in[0,2\pi),\quad j=1,\ldots,m-l.
\end{equation*}
This implies that the operator $T$ can be represented in terms of shift operators as follows
\begin{equation}\label{e7} 
T=  S_{h}^l \prod_{k=1}^{m-l}(S_{h}- e^{i\theta_k}I)
\end{equation}
with $\arg h=\varphi\pm\dfrac{\pi}2$.

Our proof of Condition $2)$ of Theorem~\ref{th:mth1} is based on the 
following simple fact.

\begin{lemma}\label{Lemma.step.1}
Let $T=S_{h}-e^{i\theta}\, I$ with $\arg h=\dfrac{\pi}2+\varphi$ for some $\varphi\in[0,\pi)$, and $\theta\in[0,2\pi)$.
If all the roots of a polynomial $p$ lie on a straight line $L_{\varphi,c_1}$ for some $c_1\in\mathbb{C}$, 
then all the roots of the polynomial~ $T(p)$ lie on the line $L_{\varphi,c_2}$ with
\begin{equation}\label{Lemma.step.1.step.formula}
c_2=c_1+\dfrac{h}2.
\end{equation}
\end{lemma}
\begin{proof}
Let $p\in\mathbb{C}[z]$ be a polynomial with all roots on a line $L_{\varphi,c_1}$. Then it can be represented as follows
\begin{equation}\label{poly.roots.on.line}
p(z)=a_0\prod\limits_{j=1}^n (z-d_je^{i\varphi} -c_1)
\end{equation}
where $a_0\neq 0$, and $d_j\in\mathbb{R}$, $j=1,\ldots,n$. If $z_0\in\mathbb{C}$ is a zero of the polynomial
$(S_{h}- e^{i\theta}I)(p)$, then one has
\begin{equation*}
(S_{h}- e^{i\theta}I)(p)(z_0)=0 \ \Longleftrightarrow \  
\prod_{j=1}^n\dfrac{e^{-i\varphi}z_0-d_j -e^{-i\varphi}c_1-i\,|h|}{e^{-i\varphi}z_0-d_j -e^{-i\varphi}c_1} = e^{i\theta}.
\end{equation*}

It is easy to see that for any $j=1, 2, \ldots , n$,  the following inequalities hold
\begin{equation*}
\left |\dfrac{e^{-i\varphi}z-d_j -e^{-i\varphi}c_1-i\,|h|}{e^{-i\varphi}z-d_j -e^{-i\varphi}c_1}\right | <1 \quad \text{whenever} \quad \Im\left(e^{-i\varphi}z\right)>\Im\left(e^{-i\varphi}c_1\right)+ \dfrac{|h|}2,
\end{equation*}
and 
\begin{equation*}
\left |\dfrac{e^{-i\varphi}z-d_j -e^{-i\varphi}c_1-i\,|h|}{e^{-i\varphi}z-d_j -e^{-i\varphi}c_1}\right | >1 \quad \text{whenever} \quad \Im\left(e^{-i\varphi}z\right)<\Im\left(e^{-i\varphi}c_1\right)+ \dfrac{|h|}2.
\end{equation*}
Consequently, all the roots of the polynomial $(S_{h}- e^{i\theta}I)(p)$ lie on the  
line 
\begin{equation}\label{Lemma.step.1.proof.1}
\left\{ z\in\mathbb{C}\ : \  \Im\left(e^{-i\varphi}z\right)=\Im\left(e^{-i\varphi}c_1\right)+ \dfrac{|h|}2 \right\}
\end{equation}
provided all the zeros  of $p$ lie on the line $L_{\varphi,c_1}$. The line~\eqref{Lemma.step.1.proof.1}
is exactly the line $L_{\varphi,c_2}$ with $c_2$ given by the formula~\eqref{Lemma.step.1.step.formula}, as required.
\end{proof}

From~\eqref{Lemma.step.1.step.formula} and~\eqref{Lemma.step.1.proof.1}, it follows that if we fix the numbers $c_1,c_2\in\mathbb{C}$ such that the operator $S_{h}-e^{i\theta}\, I$, $\theta\in[0,2\pi)$, sends any polynomial with roots on the line $L_{\varphi,c_1}$ to a polynomial with roots on the line $L_{\varphi,c_2}$ for some $\varphi\in[0,\pi)$, then
\begin{equation}
h=2\Im(e^{-i\varphi}(c_2-c_1)).
\end{equation}
In particular, $S_{h}-e^{i\theta}\, I$ preserves the set of polynomials with roots on the line $L_{\varphi,c_1}$ if and only if $h=0$. 

Thus, we obtain that the linear finite difference operator $T$ defined in~\eqref{e3.double} (see also~\eqref{e7}) with the step $h=i|h|e^{i\varphi}$ sends polynomials with roots on the line $L_{\varphi,c_1}$ to polynomials with roots on the line $L_{\varphi,c_2}$ where
\begin{equation*}
c_2=c_1+\dfrac{m-l}2\,|h|+l\,|h|=c_1+\dfrac{m+l}2\,|h|.
\end{equation*}
This means that if $c_1\neq c_2$ in the statement of Theorem~\ref{th:mth1}, then the formula in Condition $2)$ of Theorem~\ref{th:mth1} holds. At the same time, if $c_1=c_2$, that is, if the operator $T$ preserves the set of polynomials with roots on the line $L_{\varphi,c_1}$, then $m=-l$ and $h$ is an arbitrary complex number satisfying the condition $\arg h=\varphi\pm\dfrac{\pi}2$.

So the necessity of Conditions $1)$--$3)$ of Theorem~\ref{th:mth1} and Corollary~\ref{Corol.Hermite.Poulain} are proved, and we remind to the reader that Condition $4)$ of Corollary~\ref{Corol.Hermite.Poulain}  is necessary to provide the reality of the coefficients of $T(p)$.

The sufficiency of the Conditions $1)$--$3)$ of Theorem~\ref{th:mth1} and Conditions $1)$--$4)$ of Corollary~\ref{Corol.Hermite.Poulain} can be easily verified using Theorem~\ref{Theorem.de_Bruijn} and
the formula~\eqref{e7} with the corresponding change variables.

\vspace{2mm}

Let us finish this section with studying the multiplicities of the roots of polynomials $T(p)$. 

\begin{theorem}\label{th:simplicity.poly}
Under the conditions of Theorem~\ref{th:mth1}, the roots of $T(p)$ are all of multiplicity one.
\end{theorem}
\begin{proof}
Due to the formula~\eqref{e7}, it is enough to prove that if $p(z)$ has all its roots on a line $L_{\varphi,c_1}$, $c_1\in\mathbb{C}$, $\varphi\in[0,\pi)$, and $h=ire^{i\varphi}$, $r>0$, then the roots of the polynomial
\begin{equation}\label{th:simplicity.poly.proof.1}
f(z):=p(z-h)-e^{i\theta}p(z)
\end{equation}
are simple for any $\theta\in[0,2\pi)$. We prove this by contradiction.

The polynomial $p$ has the form~\eqref{poly.roots.on.line}. Let $\lambda$ be a root of $f(z)$ such that
\begin{equation}\label{th:simplicity.poly.proof.2}
f(\lambda)=f'(\lambda)=0.
\end{equation}
By Lemma~\ref{Lemma.step.1}, $\lambda$ has the form
\begin{equation*}
\lambda=\alpha e^{i\varphi}+c_1+ie^{i\varphi}\cdot\dfrac{r}2
\end{equation*}
for some $\alpha\in\mathbb{R}$.

From~\eqref{th:simplicity.poly.proof.1}--\eqref{th:simplicity.poly.proof.2} it follows that
$$
\dfrac{p'(\lambda-h)}{p(\lambda-h)}=\dfrac{p'(\lambda)}{p(\lambda)},
$$
that after simple calculations gives us
\begin{equation}\label{th:simplicity.poly.proof.3}
\sum\limits_{j=1}^n\dfrac{1}{\alpha-d_j-i\cdot\frac{r}{2}}=\sum\limits_{j=1}^n\dfrac{1}{\alpha-d_j+i\cdot\frac{r}{2}}.
\end{equation}
But this identity is impossible for any $\alpha\in\mathbb{R}$, since the imaginary part of the left-hand side of~\eqref{th:simplicity.poly.proof.3} has the form
$$
\dfrac{r}2\sum\limits_{j=1}^n\dfrac{1}{(\alpha-d_j)^2+\frac{r^2}{4}},
$$
while the imaginary part of the right-hand side of~\eqref{th:simplicity.poly.proof.3} is 
$$
-\dfrac{r}2\sum\limits_{j=1}^n\dfrac{1}{(\alpha-d_j)^2+\frac{r^2}{4}},
$$
a contradiction.
\end{proof}

We finish this section with a curious example. Let us denote by $\Delta$ the forward difference operator with the step $1$:
\begin{equation*}
\Delta(p)(z)=p(z+1)-p(z),
\end{equation*}
where $p\in\mathbb{C}[z]$. It is easy to see that
\begin{equation}\label{fin.dif.n.th.2}
\Delta^m(p)(z)=\sum_{k=0}^{m}(-1)^k\binom{m}{k}p(z+m-k)=\sum_{k=0}^{m}(-1)^{m-k}\binom{m}{k}p(z+k).
\end{equation}

Consider the monomial
$$
p(z)=z^n,\quad n\in\mathbb{N},
$$
and describe the root location of the polynomial $\Delta^m(z^n)$ of degree $n-m$. By~\eqref{fin.dif.n.th.2}, one has
\begin{equation}\label{Delta.z^n}
\Delta^mz^n=\sum\limits_{k=0}^m(-1)^{m-k}\begin{pmatrix}m\\k\end{pmatrix}(z+k)^n, \quad 1\leqslant m\leqslant n-1.
\end{equation}
Now from Lemma~\ref{Lemma.step.1} ut is easy to get the following fact.
\begin{prepos}\label{Corollary.zeroes. of.finite.differences}
	All the roots of the polynomial $\Delta^mz^n$ of degree $n-m$, $1\leqslant m\leqslant n-1$, are simple and located on the line $\Re z=-\dfrac{m}2$.
\end{prepos}

Note additionally that
\begin{equation}\label{Stirling.II}
\stirling{n}{m}=\dfrac1{m!}\,\Delta^mz^n\Bigg|_{z=0}=\dfrac1{m!}\sum\limits_{k=0}^m(-1)^{m-k}\binom{m}{k}k^n, \quad 1\leqslant m\leqslant n-1,
\end{equation}
where $\stirling{n}{m}$ are the Stirling numbers of the second kind, see, e.g.,~\cite{Stanley}. By change of variables, one can get a sequence of polynomials  with roots on the critical line.
\begin{corol}
The roots of the polynomials
\begin{equation}\label{p_nm}
S_{nm}(z)=\dfrac1{m!}\sum\limits_{k=0}^m(-1)^k\begin{pmatrix}m\\k\end{pmatrix}(mz-k)^n, \quad 1\leqslant m\leqslant n-1,
\end{equation}
of degree $n-m$ are simple, symmetric w.r.t. real axis, and located on the line $\Re z=\dfrac12$.
\end{corol}
From~\eqref{Delta.z^n}--\eqref{p_nm} we have
\begin{equation*}
\displaystyle\stirling{n}{m}=S_{nm}(1)=(-1)^{n-m}S_{nm}(0).
\end{equation*}

Note that the polynomials $\Delta^mz^n$ have the form
\begin{equation*}
\Delta^mz^n=m!\sum_{k=0}^{n-m}\binom{n}{k}\displaystyle\stirling{n-k}{m}z^k,
\end{equation*}
and satisfy the following recurrence relations
\begin{equation*}
\Delta^mz^n=(z+m)\Delta^mz^{n-1}+m\Delta^{m-1}z^{n-1}
\end{equation*}

Moreover, the zeros $\lambda_k$, $k=1,\ldots,n-1$ of the polynomial
\begin{equation*}
\Delta z^n=\sum_{k=0}^{n-1}\binom{n}{j}z^j
\end{equation*}
can be found explicitly (see~\eqref{cot.root.0} and Lemma~\ref{Lemma.roots.Delta.x^n}):
\begin{equation*}
\lambda_k=-\dfrac12-\dfrac{i}2\cot\dfrac{\pi k}n,\qquad k=1,\ldots,n-1,
\end{equation*}
so these formul\ae\ agree with Proposition~\ref{Corollary.zeroes. of.finite.differences}.

\setcounter{equation}{0}
\section{Entire functions with roots on a line}\label{section:entire_functions}

In the rest of the paper, we consider the operator $\Delta_{\theta,h}$ defined in~\eqref{Delta.theta}:
\begin{equation}\label{Delta.theta.double}
\Delta_{\theta,h}f(x)=\dfrac{e^{i\theta} f(x+ih)-e^{-i\theta} f(x-ih)}{2i},
\end{equation}
with $\theta\in[0,2\pi)$ and $h\in\mathbb{C}\setminus\{0\}$.

This section is devoted to the image of the operator $\Delta_{\theta, h}$ on a special set of entire functions to be described below. At first, let us consider the Laguerre-P\'olya class $\mathcal{LP}$ of real entire functions, see Definition~\ref{th:d2}, and prove Theorem~\ref{th:mth3}. Note that if $f \in \mathcal{LP}$, then by Theorem~\ref{Theorem.de_Bruijn}, all the zeros of $\Delta_{\theta, h}(f)$ are real. This also can be proved directly by the limiting considerations,
since $f$ is the uniform limit, on compact subsets of
$ \mathbb{C}$, of polynomials with only real zeros~\cite{PS}, and since $\Delta_{\theta, h}:
\mathcal{HP} \to \mathcal{HP}$ by Corollary~\ref{Corol.Hermite.Poulain}. However, the technique  use in the proof
of Theorem~\ref{th:mth1} does not allow us to extend it to entire functions, so we cannot assert that  the necessity
of the conditions of Corollary~\ref{Corol.Hermite.Poulain} for entire functions. Also we cannot use the limits to extend the result of Theorem~\ref{th:simplicity.poly} to entire functions, but we can prove this fact directly as in Theorem~\ref{th:simplicity.poly}.

\begin{proof}[Proof of Theorem~\ref{th:mth3}]
Without loss of generality, one can suppose that $h>0$. Let $f \in \mathcal{LP}$, $f \not\equiv 0$. As we mentioned above, the reality of zeros of $\Delta_{\theta, h}(f)$  follows from Theorem~\ref{Theorem.de_Bruijn}, so it suffices to establish the simplicity of these zeros.

Suppose that $x_0 \in \mathbb{R}$ is a multiple root of $g(z) := \Delta_{\theta, h}(f)(z)$. Then $g(x_0)=0$ and $g' (x_0)=0$, or,
equivalently,
\begin{equation*}
e^{i\theta} f(x_0 +ih) = e^{-i\theta} f(x_0 -ih),\quad   e^{i\theta} f^{\prime}(x_0 +ih) = e^{-i\theta} f^{\prime}(x_0 -ih),
\end{equation*}
so 
\begin{equation*}
\frac{f^{\prime}(x_0 +i h)}{f(x_0 +i h)} = \frac{f^{\prime}(x_0 -i h)}{f(x_0 -i h)}.
\end{equation*}
Since the function $f$ can be represented as in~\eqref{e2},  we have
\begin{equation*}
\frac{f^{\prime}(z)}{f(z)} = \frac{n}{z} -2az +b + \sum_{k=1}^\infty \frac{z}{x_k(z-x_k)},
\end{equation*}
hence we obtain
\begin{equation*}
  \frac{n}{x_0 +i h} -2a(x_0 +i h) +b + \sum_{k=1}^\infty \frac{x_0 +i h}{x_k(x_0 +i h-x_k)}=
\frac{n}{x_0 -i h} -2a(x_0 -i h) +b + \sum_{k=1}^\infty \frac{x_0 -i h}{x_k(x_0 -i h-x_k)}, 
\end{equation*}
or
\begin{equation}\label{Proof.1.simplicity.ent.func}
\frac{n}{x_0 +i h} -2aih + \sum_{k=1}^\infty \frac{x_0 +i h}{x_k(x_0 +i h-x_k)}=
\frac{n}{x_0 -i h} +2aih+ \sum_{k=1}^\infty \frac{x_0 -i h}{x_k(x_0 -i h-x_k)}.
\end{equation}

Comparing the imaginary parts of the left-hand side and the right-hand side of~\eqref{Proof.1.simplicity.ent.func},
one gets
\begin{equation*}
\frac{-nh}{x_0^2 +h^2}- 2ah - \sum_{k=1}^\infty \frac{x_k h}{x_k((x_0 -x_k)^2 + h^2)}=
\frac{nh}{x_0^2 +h^2} + 2ah + \sum_{k=1}^\infty \frac{x_k h}{x_k((x_0 -x_k)^2 + h^2)}
\end{equation*}
that implies
\begin{equation*}
\frac{n}{x_0^2 +h^2} + 2a + \sum_{k=1}^\infty \frac{1 }{(x_0 -x_k)^2 + h^2} = 0, 
\end{equation*}
since $h\neq 0$.

By~\eqref{e2}, $n\geqslant 0$, $a\geqslant 0$, and $(x_0-x_k)^2 \geqslant 0$, $k\in\mathbb{N}$. Therefore,  $f$ must be a constant function, so $\Delta_{\theta, h}(f)$ is a constant function, as well, so it cannot have multiple roots, a contradiction.
\end{proof}

Consider now a straight line $L_{\varphi,c}$ for some $\varphi\in[0,\pi)$, $c\in\mathbb{C}$, and introduce the following
class of entire functions.

\begin{definition}\label{Def.LP.line} 
A real entire function $f$ is said to be in the {\it extended
Laguerre-P\'olya class}, denoted as~$\mathcal{LP}(L_{\varphi,c})$, if
it is the limit, on compact subsets of the complex plane, of a sequence of 
polynomials with roots on the line $L_{\varphi,c}$.
\end{definition}

It is easy to see that if $f(z)\in\mathcal{LP}(L_{\varphi,c})$, then 
\begin{equation*}
g(x):=f\left(e^{i\varphi}x+c\right)\in\mathcal{LP},
\end{equation*}
so from Theorem~\ref{th:mth3} and from the formula~\eqref{e7} it is easy to obtain the following fact.

\begin{theorem}\label{th:simplicity.ent.func}
Let the function $f$ be in the class $\mathcal{LP}(L_{\varphi,c})$. If the linear finite difference operator $T$ defined in~\eqref{e3} satisfies Conditions $1)$--$3)$ of Corollary~\ref{Corol.Hermite.Poulain}, then all the zeros of $T(f)$ lie on the same line~$L_{\varphi,c}$ and are of multiplicity one.
\end{theorem}

To illustrate Theorem~\ref{th:mth3}, let us consider the reciprocal gamma function
\begin{equation*}
f(z)=\dfrac1{\Gamma(z)},
\end{equation*}
then by Theorem~\ref{th:mth3}, all the roots of the entire the function
\begin{equation}\label{recipr.gamma.delta}
g(z)=\dfrac1{2i}\left[f\left(z+\frac{i}2\right)-f\left(z-\frac{i}2\right)\right]=
\dfrac1{2i}\cdot\dfrac{\Gamma\left(z-\frac{i}2\right)-\Gamma\left(z+\frac{i}2\right)}{\Gamma\left(z+\frac{i}2\right)\Gamma\left(z-\frac{i}2\right)}
\end{equation}
are real and of multiplicity one. This fact implies, for example, the following simple theorem.
\begin{theorem}\label{Theorem.reciprocal.gamma}
	The following integral
	\begin{equation*}\label{function.F}
	F(z)=\int_{0}^{+\infty}x^{z-1}e^{-x}\sin\ln(\sqrt{x})dx
	\end{equation*}
	represents (after analytic continuation) a meromorphic function on the complex plane that has only real and simple roots and (simple) poles at the points $-n\pm\dfrac i2$, $n=0,1,2,\ldots$.
\end{theorem}
\begin{proof}
Indeed, from~\eqref{recipr.gamma.delta} one has 
$$
\dfrac1{2i}\left[\Gamma\left(z+\dfrac i2\right)-\Gamma\left(z-\dfrac i2\right)\right]=-g(z)\Gamma\left(z+\frac{i}2\right)\Gamma\left(z-\frac{i}2\right).
$$
This function is meromorphic. It has only simple real zeros and simple poles at the points  $-n\pm\dfrac i2$, $n=0,1,2,\ldots$.

Now	from the Euler integral representation of the gamma function we obtain
$$
\dfrac1{2i}\left[\Gamma\left(z+\dfrac i2\right)-\Gamma\left(z-\dfrac i2\right)\right]=
\dfrac1{2i}\int\limits_{0}^{+\infty}x^{z-1}e^{-x}\left(x^{\tfrac{i}2}-x^{-\tfrac{i}2}\right)dx=F(z),
$$
as required.
\end{proof}

\vspace{3mm}

\setcounter{equation}{0}
\section{Minimal mesh of finite differences, extremal roots, and Walsh convolution}\label{section:minimal.mesh}

In this section, we prove that the operator
\begin{equation}\label{Delta.theta.double.2}
\Delta_{\theta,h}p(x)=\dfrac{e^{i\theta}p(x+ih)-e^{-i\theta}p(x-ih)}{2i},\quad h>0,\quad \theta\in[0,\pi),
\end{equation}
defined in~\eqref{Delta.theta}, increases the mesh of hyperbolic polynomials. We remind that by the mesh of a hyperbolic polynomial we understand the minimal distance between the roots of this polynomial, see Definition~\ref{th:d4}.
Here we also prove Theorems~\ref{th:mth5} and~\ref{th:mth4}. But first, we study the roots $\lambda_{k,n}(\theta)$
of the polynomials~$Q_n$ defined in~\eqref{poly.Qn} that play a very important role in this Section.

\subsection{Zeros of the polynomials $Q_n$} 

The polynomials $Q_n$ are defined in~\eqref{poly.Qn}. It is easy to see that
\begin{eqnarray*}
	\Delta_{\theta,h}(x^n)& =&
	\dfrac{e^{i\theta}(x+ih)^n-e^{-i\theta}(x-ih)^n}{2i}=h^nQ_n\left(\dfrac{x}{h}\right)= \\
	& = &
	\sin\theta\cdot\sum_{k=0}^{\left[\tfrac n2\right]}(-1)^{k}\binom{n}{2k}h^{2k}x^{n-2k}+
	\cos\theta\cdot\sum_{k=0}^{\left[\tfrac {n-1}2\right]}(-1)^{k}\binom{n}{2k+1}h^{2k+1}x^{n-2k-1}.
\end{eqnarray*}

The roots of the equation $\Delta_{\theta,1}(x^n)=0$ were announced Introduction by the formul\ae~\eqref{cot.root.general}--\eqref{cot.root.0}. For completeness, we give here a detailed solution to this equation for arbitrary step $h>0$ and with some additional facts on their behaviour with respect to the parameter $\theta$.

\begin{lemma}\label{Lemma.roots.Delta.x^n}
The zeroes of the polynomial $\Delta_{\theta,h}(x^n)$ are real and simple and have the form
	
\noindent for $\theta\in\left(0,\pi\right)$,
\begin{equation}\label{cot.root.general.h}
\lambda_{k,n}(\theta,h)=h\cdot\cot\dfrac{\pi k-\theta}{n},\qquad k=1,\ldots,n,
\end{equation}

\noindent and for $\theta=0$,
\begin{equation}\label{cot.root.h}
\lambda_{k,n}\left(0,h\right)=h\cdot\cot\dfrac{\pi k}{n},\qquad k=1,\ldots,n-1.
\end{equation}

Moreover, whenever $\theta\neq\varphi$, $\theta,\varphi\in[0,\pi)$, the roots of $\Delta_{\theta,h}(x^n)$ and
$\Delta_{\varphi,h}(x^n)$ interlace for any $h>0$. Specifically, the functions $\lambda_{k,n}(\theta)$ are increasing on $(0,\pi)$.
\end{lemma}
\begin{proof}
The equation
$$
\Delta_{\theta,h}(x^n)=0
$$
can be represented in the form
\begin{equation*}
\dfrac{(x+ih)^n}{(x-ih)^n}=e^{-2i\theta}.
\end{equation*}
For $\theta\in(0,\pi)$, this equation has exactly $n$ solutions $\lambda_{k,n}(\theta)$, $k=1,\ldots,n$, satisfying the identity
\begin{equation*}
\dfrac{(\lambda_{k,n}+ih)^n}{(\lambda_{k,n}-ih)^n}=e^{i\,\tfrac{2\pi k-2\theta}{n}},\qquad k=1,\ldots,n,
\end{equation*}
so that $\lambda_k(\theta)$ have the form~\eqref{cot.root.general.h}. For the case $\theta=0$, the proof is similar.
	
Furthermore, for $\theta\in(0,\pi)$, we have
\begin{equation*}
\dfrac{d\lambda_{k,n}(\theta,h)}{d\theta}=\dfrac{h}n\left(1+\cot^2\dfrac{\pi k-\theta}{n}\right)>0,\qquad k=1,\ldots,n,
\end{equation*}
therefore, for any $0<\theta<\varphi<\pi$, the roots of the polynomials $\Delta_{\theta,h}(x^n)$ and
$\Delta_{\varphi,h}(x^n)$ do not coincide. Moreover, it is easy to see that
$$
\lambda_{k,n}(\theta)>\lambda_{k+1,n}(\varphi)
$$
since the cotangent function is decreasing between its poles. Consequently, 
the roots of the polynomials $\Delta_{\theta,h}(x^n)$ and
$\Delta_{\varphi,h}(x^n)$ interlace for any~$h>0$. 
\end{proof}

Furthermore, it is clear that
$$
\lambda_{1,n}(\theta,h)>\lambda_{2,n}(\theta,h)>\ldots>\lambda_{m,n}(\theta,h),
$$
where $m=n$ if $\theta\in(0,\pi)$, and $m=n-1$ if $\theta=0$ provided $h>0$. Since the cotangent function is strictly decreasing on the 
interval $(0,\pi)$ we obtain that the largest zero $\mu_{max}(Q_n)$ of $Q_n$ equals $\lambda_{1,n}(\theta,h)$, while
the minimal zero  $\mu_{min}(Q_n)$ equals $\lambda_{m,n}(\theta,h)$ that agrees with formul\ae~\eqref{min.max.Qn.zeros}--\eqref{min.max.Qn.zeros.0}.

\vspace{2mm}

Finally, let us find the mesh of the polynomials $Q_n$. A more general fact is true.

\begin{prepos}\label{Proposition.mesh.xn}
Given the operator $\Delta_{\theta, h}$ defined by the formula~\eqref{Delta.theta.double.2}, the mesh of the 
polynomials~$\Delta_{\theta, h}(x^n)$, $n\in\mathbb{N}$, satisfies the formula~\eqref{min.mesh}.
\end{prepos}
\begin{proof}
Suppose first that $\theta\in(0,\pi)$. It is clear that if
\begin{equation*}
d_k:=\lambda_{k,n}(\theta,h)-\lambda_{k+1,n}(\theta,h),
\end{equation*}
then
\begin{equation*}\label{distance.xn}
\mesh\,\Delta_{\theta,h}(x^n)=\min\limits_{1\leqslant k\leqslant n-1} d_k.
\end{equation*}
From the formula of the difference of cotangents we obtain
\begin{equation*}
d_k=\dfrac{2h\sin\frac{\pi}{n}}{\cos\frac{\pi}{n}-\cos\frac{2\pi k+\pi -2\theta}{n}}.
\end{equation*}

Now if we introduce the numbers 
\begin{equation}\label{coeff.nu}
\nu_k=\dfrac{2\pi k+\pi -2\theta}{n},\qquad k=1,\ldots,n-1,
\end{equation}
and the function
\begin{equation*}
F(u)=\dfrac{2h\sin\frac{\pi}{n}}{\cos\frac{\pi}{n}-\cos u},
\end{equation*}
then we get
\begin{equation*}
d_k=F(\nu_k),\qquad k=1,\ldots, n-1.
\end{equation*}

The function $F(u)$ is convex function on the interval $\left(\dfrac{\pi}{n},2\pi-\dfrac{\pi}{n}\right)$ satisfying the the following reflection condition
\begin{equation}\label{F.reflection}
F(u)=F(2\pi-u),\qquad u\in\left(\dfrac{\pi}{n},2\pi-\dfrac{\pi}{n}\right),
\end{equation}
so that its minimum is achieved at the point $u=\pi$.

Since $\theta\in(0,\pi)$ by assumption, it follows that 
\begin{equation}\label{coeff.nu.ineq}
\dfrac{2\pi k-\pi}{n}<\nu_k<\dfrac{2\pi k+\pi}{n},\qquad k=1,\ldots,n-1.
\end{equation}

\vspace{2mm}

Let $n=2l$, $l\in\mathbb{N}$. From~\eqref{coeff.nu} and~\eqref{coeff.nu.ineq} we have
\begin{equation*}\label{Proposition.mesh.xn.proof.1}
\dfrac{\pi}{2l}<\nu_1<\nu_2<\ldots<\nu_{l-1}<\pi-\dfrac{\pi}{2l},
\end{equation*}
and
\begin{equation*}
\pi+\dfrac{\pi}{2l}<\nu_{l+1}<\nu_{l+2}<\ldots<\nu_{2l-1}<2\pi-\dfrac{\pi}{2l}.
\end{equation*}
Moreover,
\begin{equation*}
\pi-\dfrac{\pi}{2l}<\nu_l<\pi+\dfrac{\pi}{2l}.
\end{equation*}
Therefore, 
\begin{equation*}
F(\nu_k)>F(\nu_{k+1}),\qquad k=1,\ldots,l-2,
\end{equation*}
and
\begin{equation*}
F(\nu_j)<F(\nu_{j+1}),\qquad k=l+1,\ldots,2l-2.
\end{equation*}

Now if 
\begin{equation*}
\pi-\dfrac{\pi}{2l}<\nu_l\leqslant\pi,
\end{equation*}
that is, if $\theta\in\left(0,\dfrac{\pi}2\right]$, so that $\theta=\dfrac{\pi}{2}-\dfrac{\psi}{2}$, $\psi\in(0,\pi)$, then
\begin{equation*}
\pi\leqslant2\pi-\nu_l<\pi+\dfrac{\pi}{2l}<\nu_{l+1}.
\end{equation*}
So from~\eqref{coeff.nu.ineq} from the convexity of the function $F$ we obtain
\begin{equation*}
F(\nu_l)<F(\nu_{l-1}),\qquad F(\nu_l)<F(\nu_{l+1}),
\end{equation*}
consequently,
\begin{equation*}
\mesh\,\Delta_{\theta,h}(x^n)=\min\limits_{1\leqslant k\leqslant n-1}d_k=d_l=F(\nu_l)=\dfrac{h\sin\frac{\pi}{n}}{\cos\frac{\theta}{n}\cdot \cos\frac{\pi-\theta}{n}}=
\dfrac{h\sin\frac{\pi}{n}}{\cos\frac{\pi+\psi}{2n}\cdot \cos\frac{\pi-\psi}{2n}},
\end{equation*}
that agrees with~\eqref{min.mesh}. If $\theta\in\left(\dfrac{\pi}2,\pi\right)$, then in the same one can show that the $d_l$ is the minimal distance between the roots of $\Delta_{\theta, h}(x^n)$, and the formula for the mesh of $\Delta_{\theta,h}(x^n)$ is the same.

\vspace{2mm}

Let now $n=2l+1$, $l\in\mathbb{N}$.  From~\eqref{coeff.nu} and~\eqref{coeff.nu.ineq} we have
\begin{equation*}
\dfrac{\pi}{2l+1}<\nu_1<\nu_2<\ldots<\nu_{l-1}<\pi-\dfrac{2\pi}{2l+1},
\end{equation*}
\begin{equation*}
\pi+\dfrac{2\pi}{2l+1}<\nu_{l+2}<\nu_{l+2}<\ldots<\nu_{2l-1}<2\pi-\dfrac{\pi}{2l+1}.
\end{equation*}
and
\begin{equation*}
\pi-\dfrac{2\pi}{2l+1}<\nu_l<\pi<\nu_{l+1}<\pi+\dfrac{2\pi}{2l+1}.
\end{equation*}

From the convexity of the function $F$ and from its symmetry property~\eqref{F.reflection} we obtain that the minimum of $d_k$ can be achieved for $k=l$ if $2\pi-\nu_l\leqslant\nu_{l+1}$ or for $k=l+1$ if $2\pi-\nu_l\geqslant\nu_{l+1}$. The inequality
\begin{equation*}
2\pi-\nu_l\leqslant\nu_{l+1}
\end{equation*}
holds whenever $\theta\in\left(0,\frac{\pi}{2}\right]$ while the inequality
\begin{equation*}
2\pi-\nu_l\geqslant\nu_{l+1}
\end{equation*}
holds if $\theta\in\left[\frac{\pi}{2},\pi\right)$. Thus, if $\theta=\frac{\pi-\psi}{2}$ for some $\psi\in(0,\pi)$, then
\begin{equation*}
\mesh\,\Delta_{\theta,h}(x^n)=\min\limits_{1\leqslant k\leqslant n-1}d_k=d_l=F(\nu_l)=\dfrac{h\sin\frac{\pi}{n}}{\cos\frac{\pi-2\theta}{2n}\cdot \cos\frac{\pi+2\theta}{2n}}=
\dfrac{h\sin\frac{\pi}{n}}{\cos\frac{\psi}{2n}\cdot \cos\frac{2\pi-\psi}{2n}},
\end{equation*}
If $\theta=\frac{\pi+\psi}{2}$ for some $\psi\in(0,\pi)$ we obtain analogously 
\begin{equation*}
\mesh\,\Delta_{\theta,h}(x^n)=\min\limits_{1\leqslant k\leqslant n-1}d_k=d_{l+1}=F(\nu_{l+1})=\dfrac{h\sin\frac{\pi}{n}}{\cos\frac{\pi-2\theta}{2n}\cdot \cos\frac{3\pi-2\theta}{2n}}=
\dfrac{h\sin\frac{\pi}{n}}{\cos\frac{\psi}{2n}\cdot \cos\frac{2\pi-\psi}{2n}},
\end{equation*}
which agrees with formula~\eqref{min.mesh}.

The case $\theta=0$ (or $\theta=\pi$ that is the same) can be established analogously.
\end{proof}

The correspondent properties of the roots of the polynomials $Q_n$ can be obtained by putting $h=1$ in Lemma~\ref{Lemma.roots.Delta.x^n} and Proposition~\ref{Proposition.mesh.xn}.

\subsection{Walsh convolution. Proof of Theorem~\ref{th:mth5}}

To prove Theorem~\ref{th:mth5} we need a few prelimiary facts and definitions. 

\begin{definition}
Two complex polynomials $p$ and $q$ of degree $n$ are called apolar if 
\begin{equation}\label{th:r1}
\sum_{k=0}^{n}(-1)^k \ p^{(k)}(0)\cdot q^{(n-k)}(0)=0.
\end{equation}
\end{definition} 
%

The following remarkable theorem due to J.H.\,Grace (see, e.g.,  \cite[Chapter~5,~\textsection 3,~Problem~145]{PS} and~\cite{Prasolov}) states that the complex zeros of two apolar polynomials cannot be separated by a straight line or by a circle.
\begin{theorem}[Grace, 1902]\label{Theorem.Grace}
Suppose $p$ and $q$ are two apolar polynomials of degree $n \geqslant 1$. If all zeros of~$p$ lie 
in a circular region $C,$ then $q$ has at least one zero in $C$. 
\end{theorem}
We remind that a circular region is a closed or open half-plane, disk or exterior of a disk. The following object
was studied by T.\,Takagi~\cite{Takagi} in 1921 and by J.L.\,Walsh~\cite{Walsh} in 1922, see also~\cite[Section~5.3]{Rahman_Scm}. It was named after Walsh but seems to be considered by some other researchers before.

\begin{definition} [\cite{Takagi,Walsh}]
For any two complex polynomials $p$ and $q$ of degree $n$ the \textit{Walsh convolution} is a polynomial of the following form
\begin{equation}\label{th:r2}
p\boxplus q \ (x) \ = \ \sum_{k=0}^{n} p^{(k)}(0)\cdot q^{(n-k)}(x).
\end{equation}
\end{definition} 

From~\eqref{th:r1} and~\eqref{th:r2} it is easy to see that

\begin{equation}\label{th:r3}
p\boxplus q \ (x_0)=0 \qquad\Longleftrightarrow \qquad p(-x)\quad\text{and}\quad q(x+x_0)\quad\text{are apolar}.
\end{equation}

\vspace{2mm}

Moreover, for the Walsh convolution of two polynomials the following fact is true.
\begin{theorem}[Oishi 1921]\label{Theorem.Walsh}
Let $p$ and $q$ be hyperbolic polynomials of degree $n$. Then their Walsh convolution $p\boxplus q$ is also hyperbolic.

If in addition all the roots of the polynomial $p$ lie in the interval $[\alpha,\beta]$, and all the roots of the polynomial $q$ lie in the interval $[\gamma,\delta]$, then all zeros of the polynomial $p\boxplus q$ lie in the interval $[\alpha+\gamma,\beta+\delta]$.
\end{theorem}
\noindent This theorem follows from a theorem proved by T.\,Takagi~\cite{Takagi} (see also~\cite[Theorem~5.3.3]{Rahman_Scm}) as it mentioned in~\cite{Takagi}. In its turn the Tagaki theorem follows
from a theorem proved by J.L.\,Walsh in~\cite{Walsh} (see also~\cite[Theorem~3.4.1c]{Rahman_Scm}). However, in~\cite{Takagi} T.\,Takagi attributed Theorem~\ref{Theorem.Walsh} to K.\,Oishi.

For convenience of the reader we provide a proof of this theorem.
\begin{proof}[Proof of Theorem~\ref{Theorem.Walsh}]
First we prove that $p\boxplus q$ is hyperbolic whenever $p$ and $q$ are hyperbolic. Indeed, suppose
that this is not true, and there exists a number $x_0$, $\Im x_0=b\neq0$ such that
$p\boxplus q(x_0)=0$ while $p$ and $q$ are hyperbolic. Then by~\eqref{th:r3} the polynomials $p(-x)$ and $q(x+x_0)$ are apolar. Moreover, by assumption, all the roots of  $p(-x)$  lie on the line $\{\Im z=0\}$, while all the roots of $q(x+x_0)$  
lie on  the line $ \{\Im z=-b\}$.  This implies that the roots of the polynomials $p(-x)$ and $q(x+x_0)$ can be separated 
by a straight line that contradicts the Grace theorem~\ref{Theorem.Grace}. So 
$p\boxplus q$ is hyperbolic.

Suppose now that all the roots of $p$ lie in the interval $[\alpha,\beta]$, and all the roots of the polynomial $q$ lie in the interval $[\gamma,\delta]$. Let $x_0\in\mathbb{R}$ be a root of $p\boxplus q$. Then the roots of $p(-x)$
 lie in the interval $[-\beta,-\alpha]$, while the roots of $q(x+x_0)$ lie in the interval $[\gamma-x_0,\delta-x_0]$, so by  the Grace theorem there exists a point $\zeta \in \mathbb{R}$ such that 
\begin{equation*}%
-\beta\leqslant\zeta\leqslant-\alpha\qquad\text{and}\qquad\gamma-x_0\leqslant\zeta\leqslant\delta-x_0.
\end{equation*}
Consequently, 
$$
\alpha+\gamma\leqslant x_0\leqslant\beta+\delta,
$$
as required.
\end{proof}

Now we are in a position to prove Theorem~\ref{th:mth5}.

\begin{proof}[Proof of Theorem~\ref{th:mth5}] 
Let 
\begin{equation*}
p(x)=\sum_{k=0}^n \frac{1}{k!}p^{(k)}(0) x^k
\end{equation*}

be a hyperbolic polynomial. Then
\begin{equation}\label{th:r4}
\Delta_{\theta,h}(p)(x) =\sum_{k=0}^n \dfrac{1}{k!}p^{(k)}(0)\Delta_{\theta, h}(x^k)=
\dfrac{1}{n!}\sum_{k=0}^np^{(k)}(0)\cdot\dfrac{d^{\,n-k}\left(\Delta_{\theta, h}(x^n)\right)}{dx^{n-k}}=
\dfrac{1}{n!}\cdot p \boxplus\Delta_{\theta, h}(x^n).
\end{equation}

Now the statement of Theorem \ref{th:mth5} follows from Theorem~\ref{Theorem.Walsh} and from
the formula 
$$
\Delta_{\theta, h}(x^n)=h^nQ_n\left(\dfrac{x}{h}\right).
$$ 
\end{proof}

Note that the formul\ae\eqref{min.max.Qn.zeros}--\eqref{min.max.Qn.zeros.0} for $ \mu_{max}(Q_n)$ 
and $ \mu_{min}(Q_n)$ follow from~\eqref{cot.root.general.h}--\eqref{cot.root.h} for $h=1$

\subsection{Proof of Theorem~\ref{th:mth4}}

In this section, we study the mesh of polynomials in the image of the operator $\Delta_{\theta, h}$ defined in~\eqref{Delta.theta.double.2} acting on the set $\mathcal{HP}$ of all hyperbolic polynomials. Remind that the mesh of a hyperbolic polynomial is the minimal distance between its roots, see Definition~\ref{th:d4}.

One of the first results on operators increasing the mesh of hyperbolic polynomials was obtained by M.\,Riesz who proved that
the operator of differentiation increases the mesh.
\begin{theorem}[M. Riesz, 1925]\label{Theorem.Riesz}
Let $p\in {\mathcal HP}$, $\deg p \geqslant 3$. Then 
$\mesh\, (p^{\prime}) \geqslant \mesh\,(p) $. If all zeros of $p$ are simple, then 
$\mesh\,(p') > \mesh\, (p)$. 
\end{theorem}
An elementary proof of this theorem was given by A. Stoyanoff~\cite{Sto}. 

The following a remarkable result saying that
if a hyperbolicity preserver commutes with the shift operator, it does 
not decrease mesh, was established (implicitly) by by S.\,Fisk~\cite[p. 226, Lemma 8.25]{Fisk}. This result implies Riesz's theorem, in particular.

\begin{theorem}[Fisk~\cite{Fisk}]\label{Theorem.Fisk}
If  $A: {\mathcal HP} \to  {\mathcal HP}$ 
is a linear operator commuting with the shift operator $S_b$ for any
$  b\in \mathbb{R}$, that is, $A S_b = S_b A$,  then for every
$p\in {\mathcal HP}$ the following inequality holds
$$
\mesh\, (A(p)) \geqslant \mesh\, (p).
$$
\end{theorem}
\noindent Here $S_b$ is the shift operator defined in~\eqref{shift.operator}.

Note that S.Fisk formulated this theorem in another form, so it is not easy to recognize that 
Fisk's theorem is the statement above. To make our work more self-contained, we provide a proof of 
Theorem~\ref{Theorem.Fisk} here. To do it we remind to the reader one more definition.

\begin{definition}
Let $g$ and $h$ be two hyperbolic polynomials of degree $n$ with roots $\nu_j$ and $\xi_j$, $j=1,\ldots,n$, respectively.
The polynomials $h$ and $g$ are said to have non-strictly interlacing roots if 
$$
\nu_1\leqslant\xi_1\leqslant\nu_2\leqslant\xi_2\leqslant\cdots\leqslant\nu_n\leqslant\xi_n,
$$
or
$$
\xi_1\leqslant\nu_1\leqslant\xi_2\leqslant\nu_2\leqslant\cdots\leqslant\xi_n\leqslant\nu_n.
$$
\end{definition}

The following theorem is sometimes called Obreschkov's theorem (see, for example \cite[p. 12]{O}). In fact, this theorem  rediscovered many times by different authors in the past, see the surveys~\cite{HoltzTyaglov} and~\cite{KreinNaimark}
for more details. 

\begin{theorem}\label{Theorem.pseudo.Obreschkov}
Two real polynomials $g$ and $h$ of the same degree
have (non-strictly) interlacing roots if and only if the polynomial $cg(z)+dh(z)$ is hyperbolic for any $c,d\in\mathbb{R}$.
\end{theorem}

\begin{proof}[Proof of Theorem~\ref{Theorem.Fisk}]
Let $p\in\mathcal{HP}$. If $\mesh\,(p)=0$, then the statement of the theorem is obvious (even without the condition $AS_b=S_bA$ for any $b\in\mathbb{R}$).

Suppose now that $\mesh\,(p)>0$, that is, the roots of $p$ are simple. Note that the zeros of $p(x)$ and $p(x+b)$, $b\in\mathbb{R}$, are (non-strictly) interlacing if and only if $b \leqslant\mesh\,(p)$. So according to Theorem~\ref{Theorem.pseudo.Obreschkov}, the polynomial 
\begin{equation*}
cp(x)+dS_b(p)(x)\in\mathcal{HP}
\end{equation*}
for any $c,d\in\mathbb{R}$ if and only if $b \leqslant\mesh\,(p)$.  

By assumption, $AS_b=S_bA$ for any $b\in\mathbb{R}$. Therefore, if $b \leqslant\mesh\,(p)$, then 
$cp(x)+dS_b(p)(x)\in\mathcal{HP}$ for any $c,d\in\mathbb{R}$, and the polynomial
\begin{equation*}
A(cp(x)+dS_b(p)(x))=cA(p)(x)+dA(S_b(p))(x)=cA(p)(x)+dS_b(A(p))(x)
\end{equation*}
is hyperbolic, since $A$ is a hyperbolicity preserver by assumption.

Suppose, on the contrary, that for the given polynomial $p$, the operator $A$ decreases the mesh, that is,  
$\mesh\,(p)<\mesh\,(A(p))$. Then there exists a real number $h$ such that
\begin{equation*}
\mesh\,(A(p))<h<\mesh\,(p).
\end{equation*}
Consequently, the roots of the polynomials $A(p)(x)$ and $A(p)(x+h)=S_h(A(p))(x)$ do not interlace, so 
by Theorem~\ref{Theorem.pseudo.Obreschkov} there exist real numbers $\widehat{c}$ and $\widehat{d}$ 
such that
\begin{equation*}
\widehat{c}A(p)(x)+\widehat{d}S_h(A(p))(x)\notin\mathcal{HP}.
\end{equation*}
This contradicts to the assumptions that $A$ is a hyperbolicity preserver commuting with $S_b$ for any $b\in\mathbb{R}$.
Therefore, $\mesh\,(A(p))\geqslant\mesh\,(p)$, as required.
\end{proof}
Note that in~\cite{KSV} an analogous proof was used for the minimal quotent of roots 
instead of the minimal distance. Furthermore, by Lemma~\ref{Lemma.step.1}, the linear operator $\Delta_{\theta, h}$ is a hyperbolicity preserver for any $\theta\in[0,\pi)$ and $h\in\mathbb{R}$. Since
$\Delta_{\theta, h}$  commutes with any shift operator,  it follows from Theorem~\ref{Theorem.Fisk} that $\Delta_{\theta, h}$ does not decrease mesh. 

Let us show now that in the class of all hyperbolic polynomials of degree $n$ the polynomial $x^n$ is 
extremal in the sense that the mesh of $\Delta_{\theta, h}(x^n)$ is minimal among all the meshes in the image of
the operator $\Delta_{\theta, h}$ acting on hyperbolic polynomials of degree $n$. To do this we need the following theorem.
\begin{theorem}
Let $p$ and $q$ be hyperbolic polynomials of degree $n$. Then 
\begin{equation}\label{th:r5} 
\mesh\,(p\boxplus q)\geqslant\max(\mesh\,(p),\mesh\,(q)).
\end{equation}
\end{theorem}
\begin{proof}
For a given complex polynomial, introduce the  operator $A_p:\mathbb{C}[x]\mapsto\mathbb{C}[x]$
defined by the following formula
\begin{equation*}
A_{p}(q)= p \boxplus q.
\end{equation*}
If $p\in\mathcal{HP}$, then $A_p$ is a hyperbolicity preserver by Theorem~\ref{Theorem.Walsh}, so it does not decrease mesh,
according to Theorem~\ref{Theorem.Fisk}
\begin{equation*}
\mesh\,(p\boxplus q)\geqslant\mesh\,(q).
\end{equation*}
On the other hand, the Takagi-Walsh convolution possesses the following commutative property, see, e.g., formula (5.3.3) in~\cite{Rahman_Scm},
\begin{equation*}
p\boxplus q=q\boxplus p.
\end{equation*}
Therefore, one has
\begin{equation*}
\mesh\,(p\boxplus q)\geqslant\mesh\,(p),
\end{equation*}
as required.
\end{proof}

Now we are in a position to prove Theorem~\ref{th:mth4}.

\begin{proof}[Proof of Theorem~\ref{th:mth4}]
Let us apply the operator $\Delta_{\theta, h}$ to a hyperbolic polynomial $p$. By
\eqref{th:r4} and \eqref{th:r5} one obtains
$$
\mesh\,(\Delta_{\theta, h}(p))\geqslant 
\mesh\,(p \boxplus\Delta_{\theta, h}(x^n))\geqslant\max(\mesh\, (p),\mesh\,(\Delta_{\theta, h}(x^n))),
$$
as required.
\end{proof}

Due to the commutativity of the operator $\Delta_{\theta, h}$ with any shift operator $S_b$ for $b\in\mathbb{R}$,
we have that $\mesh\,(\Delta_{\theta, h}(x^n))=\mesh\,(\Delta_{\theta, h}((x-c)^n))$ for any $c\in\mathbb{R}$.
However, our proof of Theorem~\ref{th:mth4} does not answer the question whether there exists or not
a polynomial $q\in\mathcal{HP}$ of some degree $n$, $q(x)\not\equiv(x-c)^n$ for any $c\in\mathbb{R}$, 
such that $\mesh\,(q)=\mesh\,(\Delta_{\theta, h}(x^n))$, see Section~\ref{section:open.problems} for details.

On some possibilities of extending the results of this section to the Laguerre-P\'olya class, see~\cite{Golitsyna}.

\setcounter{equation}{0}
\section{Polynomials with roots in a strip and in a half-plane}\label{section:polynomials.strip}

In this section, we prove Theorem~\ref{th:mth2}, an analogue of the 
Hermite-Poulain theorem for polynomials with roots in a strip, and also show that if a polynomial $p$ has 
roots in a half-plane, then the roots $\Delta_{\theta, ih}(p)$ can be multiple.

\begin{proof}[Proof of Theorem~\ref{th:mth2}]
First, suppose that the roots of a given polynomial $p$ lie in the strip 
$$
\mathcal{S}_r:=\{z\in\mathbb{C}\ :\ |\Im z|\leqslant r\}
$$
for some $r>0$, but not on the same line. Thus, the condition $1)$ means that the step $h$ is pure imaginary.

In this case, the proof of the necessity of the conditions $1)$ and $3)$ is similar to the proof of Theorem~\ref{th:mth1}.
Indeed, suppose that the operator~\eqref{e3} preserves the set of complex polynomials with zeros in 
the strip~$\mathcal{S}_r$. Then from \eqref{f2}  we get that all the zeros of the rational 
function 
$$
R_n(x) =\frac{T(P_n) (x)}{x^n}
$$
belong to $\mathcal{S}_r$. Then all the zeros of the function
the polynomial
$$
G_n(y) := R_n \left(\frac{n}{y}\right)
$$ 
lie in the the set
$$
D_n :=\left \{z \in\mathbb{C}:\  
\left|z + i\frac{n}{2r}\right| \geqslant  \frac{n}{2r},\    \left|z - i\frac{n}{2r}\right| \geqslant  \frac{n}{2r}\right \}.
$$ 

As $n\to\infty$, the sequence of polynomials $\{G_n(y)\}_{n=1}^{\infty}$ converges uniformly on compact sets
to the following entire function
$$
f(y) := \sum_{j=l}^m 
a_j e^{-j h y} = Q (e^{- h y}).
$$
Each zero of the function $f$ is the accumulation point of a sequence of zeros of $G_n$. 
Clearly,  if a sequence $\{ z_k\}_{k\in \mathbb{N}}$ has a limit $y_0$ and  
if	$z_k \in D_k $, $\forall k\in \mathbb{N}$,  then $y_0$  is real. In the same way as in the 
proof of Theorem~\ref{th:mth1}, we get that $\Re h =0$, and any non-zero root of $Q$ lie on the unit circle. 
Thus, the operator~$T$ defined in~\eqref{e3} has the form
\begin{equation}\label{e7e} 
T=  S_{i \beta}^l \prod_{k=1}^{m-l}(S_{i \beta}- e^{i\theta_k}I),
\end{equation}
 where $i\beta=h$, $\beta\in\mathbb{R}\setminus\{0\}$, $\theta_k\in[0,2\pi)$, $k=1,2,\ldots,m-l$, and the shift operator $S_{\lambda}$
 is defined in~\eqref{shift.operator}.
 
 Let us now prove the necessity of the condition $2)$. Suppose that all roots of a polynomial~$p$
lie on the line~$L_{0,ir}$. Then by Lemma~\ref{Lemma.step.1}, from~\eqref{e7e} we have that the roots 
of the polynomial $T(p)$ lie on the line~$L_{0,ic}$ where 
$$
c_1=r+l\beta+\dfrac{(m-l)\beta}2=r+\dfrac{(m+l)\beta}2.
$$
Analogously, if the polynomial $p$ has all roots on the line~$L_{0,-ir}$, then all roots of $T(p)$
lie on the line~$L_{0,ic_2}$ with
$$
c_2=-r+\dfrac{(m+l)\beta}2.
$$
Since the operator $T$ preserves the strip $\mathcal{S}_r$ by assumption, we must have
$$
\left|\pm r+\dfrac{(m+l)\beta}2\right|\leqslant r,
$$
that implies
$$
\dfrac{(m+l)\beta}2=0,
$$
or, equivalently, $m=-l$ as required.

\vspace{2mm}

Let us prove now the sufficiency of the conditions $1)$--$3)$ for the operator $T$ to preserve
the strip $\mathcal{S}_r$. Suppose again that all roots of the polynomial
$$
p(z)=a\prod\limits_{j=1}^n(z-z_j), \qquad a\neq0,
$$
lie in the strip $\mathcal{S}_r$, that is, $|\Im z_j|\leqslant r$, $j=1,\ldots,n$. Consider the polynomial
$$
q(z)=(S_{i\beta}-e^{i\theta}I)(p)(z),
$$
where $\beta\in\mathbb{R}\setminus\{0\}$ and $\theta\in[0,2\pi)$.

It is clear that
\begin{equation}\label{Proof.strip.1}
q(\mu)=0\quad\Longleftrightarrow\quad\prod\limits_{j=1}^n\dfrac{\mu-z_j-i\beta}{\mu-z_j}=e^{i\theta}.
\end{equation}

Since
$$
\left|\dfrac{z_0-z_j-i\beta}{z_0-z_j}\right|<1\quad\text{whenever} \Im z>r+\dfrac{\beta}2
$$
and
$$
\left|\dfrac{z_0-z_j-i\beta}{z_0-z_j}\right|>1\quad\text{whenever} \Im z<-r+\dfrac{\beta}2,
$$
one obtains that all the roots of $q(z)$ lie in the strip 
$$
\left\{z\in\mathbb{C}\ :\ \left|\Im z-\dfrac{\beta}2\right|\leqslant r\right\}.
$$
Now since the operator $T$ satisfying the conditions $1)$--$3)$ can be represented in the form~\eqref{e7e} with $m=-l$, we obtain that all the roots of the polynomial $T(p)$ lie in the strip
$$
\left\{z\in\mathbb{C}\ :\ \left|\Im z-\dfrac{\beta(m+l)}2\right|\leqslant r\right\}=S_r,
$$
as required.

Prove now that the roots of the polynomial $T(p)$ are simple provided the width of the strip is less than~$\dfrac{|h|}2$ and the operator $T$ satisfies the conditions $1)$--$3)$. Due to the formula~\eqref{e7e}, it is enough to prove that if the roots of a polynomial $p$ lie in the strip~$\mathcal{S}_r$
with $r\leqslant\dfrac{\beta}4$, $\beta>0$, then the roots of the polynomial
\begin{equation}\label{th:simplicity.poly.strip.proof.1}
f(z):=p(z-i\beta)-e^{i\theta}p(z),\quad \theta\in[0,2\pi),
\end{equation}
are simple. We prove this by contradiction.

Let
$$
p(z)=\prod\limits_{j=1}^{n}(z-z_j)
$$
with 
\begin{equation}\label{th:simplicity.poly.strip.proof.1.5}
|\Im z_j|\leqslant r\leqslant\dfrac{\beta}4,\quad j=1,\ldots,n.
\end{equation}
Let $\lambda$ be a root of $f(z)$ such that
\begin{equation}\label{th:simplicity.poly.strip.proof.2}
f(\lambda)=f'(\lambda)=0.
\end{equation}
Since
$$
f\left(z+\dfrac{i\beta}2\right)=-e^{i\tfrac{\theta}2}\left[e^{i\tfrac{\theta}2}\cdot p\left(z+\dfrac{i\beta}2\right)-e^{-i\tfrac{\theta}2}\cdot p\left(z-\dfrac{i\beta}2\right)\right],
$$
we have that
\begin{equation}\label{th:simplicity.poly.strip.proof.2.5}
\left|\Im\lambda-\dfrac{\beta}2\right|\leqslant r\leqslant\dfrac{\beta}4.
\end{equation}

From~\eqref{th:simplicity.poly.strip.proof.1}--\eqref{th:simplicity.poly.strip.proof.2} it follows that
$$
\dfrac{p'(\lambda-i\beta)}{p(\lambda-i\beta)}=\dfrac{p'(\lambda)}{p(\lambda)},
$$
that after simple calculations gives us
\begin{equation}\label{th:simplicity.poly.strip.proof.3}
\sum\limits_{j=1}^n\dfrac{1}{\lambda-z_j-i\beta}=\sum\limits_{j=1}^n\dfrac{1}{\lambda-z_j}.
\end{equation}
But this identity is impossible, since the imaginary part of the left-hand side of~\eqref{th:simplicity.poly.strip.proof.3} has the form
$$
\sum\limits_{j=1}^n\dfrac{\Im z_j-\Im\lambda+\beta}{(\Re\lambda-\Re z_j)^2+(\Im\lambda-\beta-\Im z_j)^2}>0,
$$
while the imaginary part of the right-hand side of~\eqref{th:simplicity.poly.strip.proof.3} is 
$$
\sum\limits_{j=1}^n\dfrac{\Im z_j-\Im\lambda}{(\Re\lambda-\Re z_j)^2+(\Im\lambda-\Im z_j)^2}<0.
$$
These inequalities follow from the assumptions that not all roots of $p$ lie on the same line, and from the inequalities
$$
-\beta\leqslant\Im z_j-\Im\lambda\leqslant0, \quad j=1,\ldots,n,
$$
implied by~\eqref{th:simplicity.poly.strip.proof.1.5} and~\eqref{th:simplicity.poly.strip.proof.2.5}, 
a contradiction.

Thus, Theorem~\ref{th:mth2} is true for the strip $\mathcal{S}_r$, $r>0$. The general case of an arbitrary strip bounded by lines $L_{\varphi,c_1}$ and $L_{\varphi,c_2}$, $c_1\neq c_2$, can be obtained from this particular case by changing of variables $w=e^{-i\varphi}z+\dfrac{i}2\cdot\Im\left(e^{-i\varphi}(c_1+c_2)\right)$.
\end{proof}

Unfortunately, we cannot say whether the restriction on the width of the strip is sharp in Theorem~\ref{th:mth2}. Calculations show that the roots of $T(p)$ are simple if the width of the strip where the roots of the polynomial $p$ lie
is less than $2|h|$, and we suppose that this is true, see Section~\ref{section:open.problems} for details.

\vspace{2mm}

As we mentioned in Introduction, zero strip preservers are also zero half-plane preservers, see Corollary~\ref{Corol:mth2.}. However, in the case of half-planes, for any step $h$ there exists a polynomial $p_h$ such that the roots of the polynomial $T(p_h)$ are multiple. For example, for the polynomial $p_h(x)=(x^2+h^2)^2$, we have
$
p(x+ih)-p(x-ih)=x^3.
$

\setcounter{equation}{0}
\section{Asymptotics of the roots of $\Delta_{\theta,h}$}\label{section:asymptotics}

In this section, we prove Theorem~\ref{th:mth6}  stating that if
\begin{equation}\label{polynomial.general}
p(z)=a_0z^n+a_1z^{n-1}+\cdots+a_{n-1}z+a_n,\qquad a_j\in\mathbb{C},\  a_0\neq0,
\end{equation}
is an arbitrary complex polynomial, then the $k$-th root of the polynomial
$\Delta_{\theta, h}(p)$ satisfies the following asymptotic formula
\begin{equation}\label{asympt.formula.double}
\begin{array}{l}
\displaystyle\mu_k(\theta,h)=h\cdot\lambda_{k,n}(\theta)-\dfrac{a_1}{n\,a_0}-\dfrac{Q_n''\left(\lambda_{k,n}(\theta)\right)}{n!\,Q_n'\left(\lambda_{k,n}(\theta)\right)}\cdot
\dfrac{p^{(n-2)}\left(-\tfrac{a_1}{n\,a_0}\right)}{a_0h}-\\
 \\
\displaystyle-\dfrac{Q_n'''\left(\lambda_{k,n}(\theta)\right)}{n!\,Q_n'\left(\lambda_{k,n}(\theta)\right)}\cdot
\dfrac{p^{(n-3)}\left(-\tfrac{a_1}{n\,a_0}\right)}{a_0h^2}\,+\,
O\left(\frac{1}{h^3}\right), \quad k=1,\ldots,m,
\end{array}
\end{equation}
as $|h|\to\infty$, where
$$
Q_n(x)=\Delta_{\theta,1}(x^n)=\dfrac{e^{i\theta}(x+i)^n-e^{-i\theta}(x-i)^n}{2i},
$$
and $\lambda_{k,n}(\theta)$ are its roots defined in~\eqref{cot.root.general} and~\eqref{cot.root.0}. We remind the reader that the operator $\Delta_{\theta, h}$ is defined in~\eqref{Delta.theta}. The number $m$ equals $n$ if $\theta\in(0,2\pi)$, and $m=n-1$ if $\theta=0$. For the sake of brevity, we will denote the roots of the polynomial $Q_n(x)$ as $\lambda_{k,n}$.

Let us denote
\begin{equation}
P_n(z):=\Delta_{\theta, h}\,p(z),  
\end{equation}
where $p$ is defined in~\eqref{polynomial.general}.
Dividing $P_n(z)$ by $a_0h^n$, one obtains
$$
\dfrac{P_n(z)}{a_0h^n}=\sum\limits_{j=0}^n\dfrac{a_j}{a_0h^n}\cdot\dfrac{e^{i\theta}(z+ih)^{n-j}-e^{-i\theta}(z-ih)^{n-j}}{2i}
=\sum\limits_{j=0}^n\dfrac{a_j}{a_0h^j}\cdot Q_{n-j}\left(\dfrac{z}{h}\right).
$$

Changing variables as follows
\begin{equation*}
x=\dfrac{z}h,
\end{equation*}
we can reformulate the problem as the finding the asymptotics of the roots $\nu_{k,n}(\theta,h)$ of 
the equation
\begin{equation}\label{asympt.equation.2}
Q_n(x)+\dfrac{a_1}{a_0h}\cdot Q_{n-1}(x)+\dfrac{a_2}{a_0h^2}\cdot Q_{n-2}(x)+\sum\limits_{j=3}^n\dfrac{a_j}{a_0h^j}\cdot Q_{n-j}\left(x\right)=0
\end{equation}
as $|h|\to\infty$ for any $\theta\in[0,2\pi)$. Note that
\begin{equation}\label{deriv.Q_n}
Q_{l}(x)=\dfrac{l!}{n!}\cdot\dfrac{d^{n-l}\left[Q_{n}(x)\right]}{dx^{n-l}},\quad, l=1,\ldots,n-1,
\end{equation}
so equation~\eqref{asympt.equation.2} can be rewritten in the following form
\begin{equation}\label{asympt.equation.1}
Q_n(x)+\dfrac{a_1}{a_0\,h}\cdot \dfrac{Q'_{n}(x)}{n}+\dfrac{a_2}{a_0\,h^2}\cdot\dfrac{Q''_{n}(x)}{n(n-1)}+
\sum\limits_{j=3}^n\dfrac{a_j}{a_0\,h^j}\cdot\dfrac{(n-j)!\cdot Q^{(j)}_{n}\left(x\right)}{n!}=0.
\end{equation}

\vspace{2mm}

Let us fix an index $k$, $k=1,\ldots,m$, and introduce the polynomial
\begin{equation}\label{a4} 
g_n(x)=\dfrac{Q_{n}(x)}{x-\lambda_{k,n}}.
\end{equation}
It is easy to see that
\begin{equation*}
g_n\left(\lambda_{k,n}\right)=\lim\limits_{x\to\lambda_{k,n}}\dfrac{Q_{n}(x)}{x-\lambda_{k,n}}=
Q_n'(\lambda_{k,n}),
\end{equation*}
and, in general,
\begin{equation}\label{Function.g}
g^{(l)}_n\left(\lambda_{k,n}\right)=
\dfrac{Q_n^{(l+1)}(\lambda_{k,n})}{l+1},\qquad l=1,\ldots,n-1.
\end{equation}

By Hurwitz's theorem, there exists a number $\rho>0$ such that for large values of the number $h$ the circle $|x-\lambda_{k,n}|<\rho$ contains only one root of the equation~\eqref{asympt.equation.1}. This root is $\nu_{k,n}(\theta,h)$, that is, 
\begin{equation}\label{c1} 
|\nu_{k,n}(\theta,h)-\lambda_{k,n}|<\rho,\quad |\nu_{l,n}(\theta,h)-\lambda_{k,n}|\geqslant \rho, \ \ l\neq k. 
\end{equation}  
Therefore, $g_n(\nu_{k,n}(\theta,h))\neq0$, and we can divide~\eqref{asympt.equation.1} by $g_n(\nu_{k,n}(\theta,h))$. So the root $\nu_{k,n}(\theta,h)$ satisfies the equation 
\begin{equation}\label{a9}
\begin{array}{c}
\nu_{k,n}(\theta,h)-\lambda_{k,n}+\dfrac{a_1}{a_0\,n}\cdot \dfrac{Q'_{n}(\nu_{k,n}(\theta,h))}{g_n(\nu_{k,n}(\theta,h))}\cdot\dfrac{1}{h}+
\dfrac{a_2}{a_0\,n(n-1)}\cdot\dfrac{Q''_{n}(\nu_{k,n}(\theta,h))}{g_n(\nu_{k,n}(\theta,h))}\cdot\dfrac{1}{h^2}+\\
\\
\sum\limits_{j=3}^{n}\dfrac{a_j}{a_0}\cdot\dfrac{(n-j)!}{n!}\cdot\dfrac{Q^{(j)}_{n}(\nu_{k,n}(\theta,h))}
{g_n(\nu_{k,n}(\theta,h))}\cdot\dfrac1{h^j} =0.  
\end{array}
\end{equation} 
Now from~\eqref{c1} and~\eqref{a9} it follows that 
\begin{equation}\label {a13} 
\nu_{k,n}(\theta,h)=\lambda_{k,n}(\theta) +O\left(\frac{1}{h}\right). 
\end{equation}

Furthermore, from~\eqref{Function.g} it follows that the Taylor expansions of the functions $\dfrac{Q^{(l)}_n}{g_n}$, $l=1,2,3$,  at $\lambda_{k,n}$  have the forms
\begin{equation}\label{Taylor.1}
\begin{array}{c}
\dfrac{Q'_n(x)}{g_n(x)}=1+\dfrac{1}{2}\cdot\dfrac{Q''_n(\lambda_{k,n})}{Q_n'(\lambda_{k,n})}\cdot(x-\lambda_{k,n})+\\
\\
+\dfrac12\,\left[\dfrac23\cdot\dfrac{Q'''_n(\lambda_{k,n})}{Q_n'(\lambda_{k,n})}-\dfrac12\cdot\left(\dfrac{Q''_n(\lambda_{k,n})}{Q_n'(\lambda_{k,n})}\right)^2\right]\cdot(x-\lambda_{k,n})^2+O\left((x-\lambda_{k,n})^3\right),
\end{array}
\end{equation}
\begin{equation}\label{Taylor.2}
\dfrac{Q''_n(x)}{g_n(x)}=\dfrac{Q''_n(\lambda_{k,n})}{Q_n'(\lambda_{k,n})}+
\left[\dfrac{Q'''_n(\lambda_{k,n})}{Q_n'(\lambda_{k,n})}-\dfrac12\cdot\left(\dfrac{Q''_n(\lambda_{k,n})}{Q_n'(\lambda_{k,n})}\right)^2\right]\cdot(x-\lambda_{k,n})+O\left((x-\lambda_{k,n})^2\right),
\end{equation}
\begin{equation}\label{Taylor.3}
\dfrac{Q'''_n(x)}{g_n(x)}=\dfrac{Q'''_n(\lambda_{k,n})}{Q_n'(\lambda_{k,n})}+O(x-\lambda_{k,n}),
\end{equation}
as $x\to\lambda_{k,n}$.

If now we represent the difference $\nu_{k,n}(\theta,h)-\lambda_{k,n}(\theta)$ as follow
$$
\nu_{k,n}(\theta,h)-\lambda_{k,n}(\theta):=\dfrac{A}{h}+\dfrac{B}{h^2}+\dfrac{C}{h^3}+O\left(\dfrac1{h^4}\right),
$$
and substitute this to~\eqref{a9}, then with the expansions~\eqref{Taylor.1}--\eqref{Taylor.3}, after simplification we get
\begin{equation}\label{lambda.difference}
\begin{array}{l}
\nu_{k,n}(\theta,h)-\lambda_{k,n}(\theta)=-\dfrac{a_1}{na_0}\cdot\dfrac{1}{h}-
\dfrac{1}{a_0n(n-1)}\cdot\dfrac{Q''_n(\lambda_{k,n})}{Q'_n(\lambda_{k,n})}\cdot\left[a_2-\dfrac{n-1}{n}\cdot\dfrac{a_1^2}{2a_0}\right]\cdot\dfrac1{h^2}-\\
\\
\dfrac{1}{a_0n(n-1)(n-2)}\cdot\dfrac{Q'''_n(\lambda_{k,n})}{Q'_n(\lambda_{k,n})}\cdot
\left[a_3-\dfrac{n-2}{n}\cdot\dfrac{a_1a_2}{a_0}+\dfrac{(n-1)(n-2)}{3n^2}\cdot\dfrac{a_1^3}{a_0^2}\right]\cdot\dfrac1{h^3}+\omega(\theta,h),
\end{array}
\end{equation}
%
where
\begin{equation}\label {q5}
|\omega (\theta,h) |\leqslant \frac{K}{h^4}, \quad\text{as}\quad |h|\to\infty.
\end{equation} 
Here $K>0$ is a constant.

Note that
$$
a_2-\dfrac{n-1}{n}\cdot\dfrac{a_1^2}{2a_0}=\dfrac{1}{(n-2)!}\cdot p^{(n-2)}\left(-\dfrac{a_1}{na_0}\right)
$$
and
$$
a_3-\dfrac{n-2}{n}\cdot\dfrac{a_1a_2}{a_0}+\dfrac{(n-1)(n-2)}{3n^2}\cdot\dfrac{a_1^3}{a_0^2}=
\dfrac{1}{(n-3)!}\cdot p^{(n-3)}\left(-\dfrac{a_1}{na_0}\right).
$$
From these formul\ae\ and from~\eqref{lambda.difference} and~\eqref{q5} we obtain~\eqref{asympt.formula.double},
since 
$$
\mu_{k,n}(\theta,h)=\dfrac{\nu_{k,n}(\theta,h)}{h},
$$
as required.

\setcounter{equation}{0}
\section{Conclusion and open problems}\label{section:open.problems}

In this work, we studied properties of linear finite differences operators to preserve the certain classes of polynomials and entire functions. For the operator~\eqref{e3} we establish an analogue of the Hermite-Poulain theorem, and establish a number of remarkable properties of the operator 
\begin{equation*}\label{Delta.theta.2222}
\Delta_{\theta,h}f(x)=\dfrac{e^{i\theta} f(x+ih)-e^{-i\theta} f(x-ih)}{2i},
\end{equation*}
defined in~\eqref{Delta.theta}. We showed that this operator preserves the class $\mathcal{LP}$. More exactly, 
the operator~\eqref{Delta.theta} maps $\mathcal{LP}$ into the subclass of $\mathcal{LP}$ of functions with
only simple zeros. Moreover, the image of the class of hyperbolic polynomials $\mathcal{HP}$ is the subclass of $\mathcal{LP}$ of polynomials with simple roots with minimal mesh given by~\eqref{mesh.main.ineq}. Thus, in connection with Theorems \ref{th:mth5} and \ref{th:mth4} the following natural question arises. 

\vspace{2mm}

\noindent\textbf{Open problem.} \textit{To describe the image of the set of hyperbolic polynomials (of the set of hyperbolic	polynomials of degrees not greater than a given $n$)  under the linear operator of the form~\eqref{Delta.theta}}.

\vspace{2mm}

For example, it is easy to see that~\eqref{Delta.theta} preserves the class of the so-called self-interlacing polynomials for~$\theta=0$.
\begin{definition}[\cite{Tyaglov_self}]
	A real polynomial $p(z)$ is called \textit{self-interlacing} if it has real and simple roots, and the roots of $p(z)$
	strictly interlace the roots of $p(-z)$. The class of self-interlacing polynomials is denoted~\textbf{SI}.
\end{definition}

The operator $\Delta_{0,h}$ preserves the class~\textbf{SI}.

\begin{theorem}\label{Theorem.SI}
	If $p(z)$ is a self-interlacing polynomial, then the polynomial
	$$
	\Delta_{0,h}p(z)=\dfrac{p(z+ih)-p(z-ih)}{2ih}
	$$
	also is self-interlacing for any $h>0$.
\end{theorem}
\begin{proof}
	By Theorem~\ref{Theorem.pseudo.Obreschkov}, for any $a,b\in\mathbb{R}$, the polynomial 
	$$
	q(z)=ap(z)+bp(-z)
	$$
	as only real (and simple) roots. Let us prove that if 
	$$
	Q(z)= \dfrac{p(z+ih)-p(z-ih)}{2ih},
	$$
	then the polynomial $aQ(z)-bQ(-z)$ has only re]al roots for any $a,b\in\mathbb{R}$. 
	
	Indeed, it is easy to see that
	$$
	aQ(z)-bQ(-z)=\dfrac{ap(z+ih)-ap(z-ih)-bp(-z+ih)+bp(-z-ih)}{2ih}=\dfrac{q(z+ih)-q(z-ih)}{2ih}.
	$$
	Since $q(z)$ has only real roots, the polynomial  $aQ(z)-bQ(-z)$ has only real and simple roots by Theorem~\ref{th:mth3}. Now from Theorem~\ref{Theorem.pseudo.Obreschkov}, the polynomials $Q(z)$ and $Q(-z)$
	have real simple and interlacing roots. Therefore, $Q(z)\in\mathbf{SI}$, as required.
\end{proof}

However, the action of the operator~\eqref{Delta.theta} on the class of all complex polynomials and its subclasses is also interesting. For example, the self-interlacing polynomials have a strong connection the class of real Hurwitz stable polynomials, the polynomials with roots in the open left half-plane, see~\cite{Tyaglov_self}. The operator~$\Delta_{\theta, h}$ is a complex zero decreasing operator as it was mentioned by N.G.\,de Bruijn~\cite{Br}. At the same time, calculations show that the following conjecture can be true.
\begin{conjecture}
\textit{For any stable polynomial~$p$ the polynomials}
$$
\Delta_{\theta, h}^m(p)(z),\qquad m=1,\ldots,\deg p-1.
$$
\textit{have no nonreal zeroes in the closed right half-plane.}
\end{conjecture}
Moreover, calculations allow us to pose the following conjecture.
\begin{conjecture}
\textit{	Let $p$ be the following polynomial}
	$$
	p(x)=x^n+x^{n-1}+\cdots+x+1=\dfrac{x^{n+1}-1}{x-1},
	$$
\textit{	then the polynomial $\Delta_{\theta, h}(p)(x)$ is self-interlacing, or self-interlacing multiplied by $x$.}
\end{conjecture}

Furthermore, Theorem~\ref{th:mth3} says that if $p\in\mathcal{HP}$, then all roots of the polynomial $\Delta_{\theta,h}(p)$ are simple. Calculations show that most of the polynomials in the image of the 
operator $\Delta_{\theta,h}(p)$ have simple roots.
\begin{conjecture}
\textit{	For almost all  polynomial $p$, the polynomials $\Delta_{\theta, h}(p)$ have only simple roots. The exclusion: Polynomials with multiple roots of the form $\pm nih$, $n\in\mathbb{N}$.}
\end{conjecture}
Such polynomials have the form $p(x)=q(z)(z+inh)^m(z+i(n+2)h)^k$, where $m,k\geqslant 2$ and $q(z)$ is an arbitrary complex polynomial.

\vspace{2mm}

Finally, note that in~\cite{Br} and~\cite{Cardon}, it was proved that the operator~\eqref{Delta.theta} is zero strip 
decreasing if defined on $\mathbb{R}[z]$. Calculations show that this property is preserved for arbitrary complex polynomials.

\begin{conjecture}
\textit{The operator~\eqref{Delta.theta} defined on $\mathbb{C}[z]$ zero strip decreasing.}
\end{conjecture}

\section*{Acknowledgement}

\noindent The authors are grateful to J. Xia for his help with preparing the first version of the manuscript.

The work of M.\,Tyaglov was partially supported by The Program for Professor of Special Appointment (Oriental Scholar) at Shanghai Institutions of
Higher Learning, by the Joint NSFC-ISF Research Program, jointly funded by the National Natural Science Foundation of China and the
Israel Science Foundation (No.11561141001), and by Grant AF0710021 from Shanghai Jiao Tong University University.

\end{document}